\documentclass[12pt]{amsart}
\usepackage{amsmath,amsfonts,latexsym,graphicx,amssymb,url}
\usepackage{amsmath,txfonts,pifont,bbding,pxfonts,manfnt}
\usepackage{hyperref}
\usepackage[active]{srcltx}
\usepackage{pdfsync}
\usepackage{hyperref}
\usepackage[active]{srcltx}
\usepackage{wasysym,pstricks, enumerate}
\usepackage{lscape,color,subfigure,graphicx}
\usepackage{enumerate}
\usepackage{enumitem}

\newcommand{\dist}{\text{dist}} 
\newcommand{\diam}{\text{diam}}

\newcommand{\R}{{\mathbb R}} 

\def \e {\varepsilon}

\def \e{\epsilon}
\renewcommand{\(}{\left(}
\renewcommand{\)}{\right)}

\theoremstyle{plain}
\newtheorem{theorem}{Theorem}[section]

\newtheorem{lemma}[theorem]{Lemma}

\newtheorem{definition}[theorem]{Definition}
\newtheorem{remark}[theorem]{Remark}

\providecommand{\bysame}{\makebox[3em]{\hrulefill}\thinspace}

\begin{document}

\setcounter{equation}{0}










\title{$C^{1,\alpha}$-estimates for the near field refractor}
	
\author[C. E. Guti\'errez and F. Tournier]
{Cristian E. Guti\'errez\\
 and \\
 Federico Tournier}
\thanks{\today\\The first author was partially supported
by NSF grant DMS--1600578.}
\address{Department of Mathematics\\Temple University\\Philadelphia, PA 19122}
\email{gutierre@temple.edu}
\address{Instituto Argentino de Matem\'atica A. P. Calder\'on, CONICET, Buenos Aires, Argentina}

\email{f-tournier@hotmail.com}

\begin{abstract} 
We establish local $C^{1,\alpha}$ estimates for one source near field refractors under structural assumptions on the target, and with no assumptions on the smoothness of the densities.
\end{abstract}
\maketitle
\setcounter{equation}{0}

\tableofcontents

\section{Introduction}
The main purpose in this paper is to prove H\"older estimates for gradients of weak solutions to the near field refractor problem introduced in \cite{gutierrez-huang:nearfieldrefractor}, where existence of weak solutions is proved as a consequence of a general abstract method applicable also in other situations.
The set up for the problem is as follows.
Suppose we have a domain $\Omega\subset S^{n-1}$ and 
a domain $\Sigma$ contained in an $n$ dimensional surface in $\R^n$; here,
$\Omega$ denotes the set of incident directions, and $\Sigma$ denotes the target domain, receiver, or screen to be illuminated.
Let $n_1$ and $n_2$ be the indices of refraction
of two homogeneous and isotropic media I and II, respectively.
From a point source at the origin, surrounded by medium I,
radiation emanates in each direction $x$ with intensity $f(x)$
for $x\in \Omega$, and the target $\Sigma$ is surrounded by medium II.
A {\it near field refractor} is an optical surface
$\mathcal R$, 
interface between media I and II,
such that all rays refracted by $\mathcal R$ into medium II, in accordance with the Snell law,
are received at the surface $\Sigma$ with prescribed radiation intensity distribution given by a measure $\nu$.
Assuming no loss of energy in this process, we have the conservation of energy equation
$\int_{\Omega}f(x)\,dx=\nu(\Sigma)$.
Under visibility assumptions on the target and conditions to avoid total reflection, existence of solutions to this problem is proved in \cite{gutierrez-huang:nearfieldrefractor}.

The problem solved in the present paper is that weak solutions are $C^1$ and their gradients are locally H\"older continuous under no smoothness assumptions on the density $f$ and the measure $\nu$. In fact, we prove a more general result, Theorem \ref{thm:c1alpha estimates for general refractors}, valid for more general near field refractors in the sense of Definition \ref{def:definition of refractor}.
Our assumptions are of structural nature, that is, they depend on the relative location of the target, its visibility from the cone of incident directions, and its convexity; see Section \ref{sec:structural assumptions on the target}.
A major difficulty with the near field refractor problem is that solutions have a complicated structure  given by Descartes ovals that often require difficult analytical estimates, and it does not have an optimal mass transport structure. 

To place our results in perspective we mention that regularity results for one source far field reflectors are in 
\cite{caffarelli-gutierrez-huang:antennaannals}, results for near field parallel refractors are in \cite{gutierrez-tournier:REGULARITYFORTHENEARFIELDPARALLELREFRACTORANDREFLECTORPROBLEMS} and \cite{abedin-gutierrez-tralli:parallelrefractor}, and results for generated Jacobian equations, including reflector problems,  are in  
\cite{2017-guillen-kitagawa:generatedjacobianeqs}.
Numerical methods are developed in \cite{deleo-gutierrez-mawi:numericalrefractor} to solve the one source far field refractor problem, in \cite{2019-gutierrez-mawi:numericalsolutionnearfield} to solve the near field, and in \cite{abedin-gutierrez:numericalgeneratedjacobians} to solve generated Jacobian equations.

The organization of the paper is as follows. Section \ref{sec:preliminaries structure and example} contains structural conditions on the target as well as a discussion on them and an example.
Analytical estimates for ovals and a maximum principle, Lemma \ref{crucial}, of the type developed in \cite{loeper:actapaper} and \cite{2010-kim-mccann:DMASM} are contained in Section \ref{sec:preliminary for ovals and max principle}.
More analytical estimates for derivatives of ovals are in Section \ref{sec:estimates for derivatives of ovals}.
Section \ref{sec:C 1 alpha estimates for near field refractors} contains the H\"older estimates,  where the main result is Theorem \ref{thm:main} from which we deduce as consequences Theorems \ref{thm:c1alpha estimates for general refractors} and  \ref{thm:regularity of Brenier solutions}.

\section{Preliminaries, structural assumptions, and examples}\label{sec:preliminaries structure and example}

\setcounter{equation}{0}

Recall that a Descartes oval is the set $\mathcal O(Y,b)=\{X\in \R^{n}:|X|+\kappa|X-Y|=b\}$, with $\kappa|Y|<b<|Y|$. Here $\kappa=n_2/n_1$, where $n_1$ is the refractive index of the material inside the oval and $n_2$ is the refractive index of the material outside. We assume throughout that $\kappa<1$, which is the most interesting from an optical point of view (when $\kappa>1$ the arguments are similar).
From the Snell law, a ray emanating from the origin with unit direction $x$ is refracted at the point $X\in \mathcal  O(Y,b)$ into the point $Y$ provided that 
\begin{equation}\label{refraction}
\dfrac{X}{|X|}\cdot \dfrac{Y-X}{|Y-X|}\geq\kappa;
\end{equation}
an inequality that by the equation of the oval is equivalent to $x\cdot Y\geq b$.
The polar equation of the oval is $\mathcal O(Y,b)=\{\rho(x,Y,b)x:x\in S^{n-1}\}$ where 
\begin{equation}\label{eq:polar radius of the oval}
\rho(x,Y,b)=\dfrac{b-\kappa^2\,x\cdot Y-\sqrt{(b-\kappa^2\, x\cdot Y)^{2}-(1-\kappa^{2})(b^{2}-\kappa^{2}|Y|^{2})}}{1-\kappa^{2}}.
\end{equation}
For a geometric analysis and estimates for Descartes ovals we refer to \cite[Sec. 4]{gutierrez-huang:nearfieldrefractor}. 
If we specify a point $X_0$ on the oval $\mathcal O(Y,b)$, then $b=|X_0|+\kappa|X_0-Y|$ and it will be useful to introduce the function 
\begin{equation}\label{eq:definition of function h}
h(x,Y,X_0)=\rho(x,Y,b),
\end{equation}
with the point $X_0$ so that $\dfrac{X_0}{|X_0|}\cdot \dfrac{Y-X_0}{|Y-X_0|}\geq\kappa$.
For $\Omega\subseteq S^{n-1}$ open and constants $0<c_1<c_2$, we let
\[
\Gamma_{c_1c_2}=\left\{rx:x\in\Omega,\,c_1\leq r\leq c_2\right\}.
\]

\subsection{Structural assumptions on the target $\Sigma$}\label{sec:structural assumptions on the target}
We begin introducing the following notion of curve in $S^{n-1}$ that will be used to state our assumptions.

Let $x_0,\hat m,\bar m\in S^{n-1}$ with 
$\bar m\cdot x_0\geq\kappa$ and $\hat m\cdot x_0\geq\kappa$.
By definition, $[\bar m,\hat m]_{x_0}$ denotes the curve obtained intersecting the triangle with vertices $\bar m$, $\hat m$, and $x_0/\kappa$ with the sphere $S^{n-1}$. Notice that since $\kappa<1$, the point $x_0/\kappa$ is outside the unit ball.
In this triangle, the side joining $\hat m$ and $\bar m$, is given by $m_{\lambda}=(1-\lambda)\bar m+\lambda\hat m$, with $0\leq \lambda\leq 1$.
Each point $m\in [\bar m,\hat m]_{x_0}$ can then be obtained intersecting the line $\dfrac{x_0}{\kappa} +\beta\,\xi$ with the sphere $S^{n-1}$, where 
$\xi=m_{\lambda}-\dfrac{1}{\kappa}x_0$, $\beta\in \R$.
Solving for $\beta$ yields 
\begin{equation}\label{eq:definition of beta of lambda}
\beta(\lambda)=\dfrac{-x_0\cdot \xi -\sqrt{(x_0\cdot \xi)^2-\(1-\kappa^{2}\)|\xi|^{2}}}{\kappa|\xi|^{2}},
\end{equation} 
since the point $\dfrac{x_0}{\kappa} +\beta\,\xi$ is inside the triangle so $0<\beta <1$.
Therefore, we obtain the parametrization
\begin{equation}\label{eq:parametrization of the x0 segment}
[\bar m,\hat m]_{x_0}=\left\{m(\lambda)=\frac{1}{\kappa}x_0+\beta(\lambda)\(m_{\lambda}-\frac{1}{\kappa}x_0\),\ \lambda\in [0,1]\right\}.
\end{equation}
In particular, for $m\in [\bar m,\hat m]_{x_0}$ we can write 
\begin{equation}\label{eq:writing of m in the curve hatm bar m}
m=\frac{1}{\kappa}x_0+\bar \beta\,\(\bar m-\frac{1}{\kappa}x_0\)+\hat \beta\,\(\hat m-\frac{1}{\kappa}x_0\)
\end{equation}
with $\bar \beta,\hat \beta\geq 0$ and $\bar \beta+\hat \beta\leq 1$; $\bar \beta=(1-\lambda)\beta(\lambda)$, $\hat \beta=\lambda\,\beta(\lambda)$.
Notice that $m(\lambda)\cdot x_0\geq \kappa$ for $0\leq \lambda\leq 1$ since $\beta(\lambda)\leq 1$ and $\kappa<1$.

We next introduce our structural assumptions.

\begin{enumerate}[label=\textbf{H.\Alph*}]
\item For each $X\in\Gamma_{c_1c_2}$, let $C_{X}=\left\{Y:\dfrac{X}{|X|}\cdot \dfrac{Y-X}{|Y-X|}\geq\kappa\right\}$ be the cone with vertex $X$, axis $X/|X|$, and opening $\arccos \kappa$. Set 
\[
\mathcal C_{\Omega}=\bigcap_{X\in\Gamma_{c_1c_2}} C_{X}. 
\]
We assume the following:
\begin{enumerate}
\item[(a)] $\Sigma\subset \mathcal C_{\Omega}$, so \eqref{refraction} holds
for all $Y\in\Sigma$ and $X\in\Gamma_{c_1c_2}$;
\item[(b)] For each $X\in \Gamma_{c_1c_2}$ there exists a set $E(X)\subset \{m\in S^{n-1}:m\cdot x\geq \kappa,x=X/|X|\}$ and a continuous function 
$s_{X}:E(X)\to \R^+$ such that 
\[
\Sigma=\{X+s_{X}(m)\,m:m\in E(X)\},
\]
with the set $E(X)$ satisfying $[\bar m,\hat m]_x\subset E(X)$ for all $\bar m,\hat m\in E(X)$, with $x=X/|X|$;
\item[(c)] The family of functions $\{s_X\}_{X\in \Gamma_{c_1c_2}}$ is uniformly Lipschitz continuous, i.e.,
there exists a constant $C>0$ such that $|s_X(m_1)-s_X(m_2)|\leq C\,|m_1-m_2|$ for all $m_1,m_2\in E(X)$ and $X\in \Gamma_{c_1c_2}$. 
\end{enumerate}
\label{item:hypotheses A}
\item Let $C(\kappa)=\kappa\(\sqrt{1+(1+\kappa)^{-2}}-1\)$.
We assume 
$
\dfrac{|X|}{|Y-X|}\leq C(\kappa)$ for all $X\in\Gamma_{c_1c_2}$ and $Y\in\Sigma$.
Notice that this holds if  
$
\dist(\Gamma,\Sigma)\geq c_2/C(\kappa).
$
\label{item:old hypotheses C} 
\item 
There exists a constant $0\leq \mu< \kappa$ such that for all $X_0\in \Gamma_{c_1c_2}$ and $\bar m,\hat m\in E(X_0)$, the function $s_{X_0}$ satisfies the following concavity condition 
\begin{equation*}\label{HC}
\dfrac{1}{s_{X_0}(m(\lambda))}+\dfrac{\mu}{|X_0|}
\geq
\bar\beta(\lambda)\,
\(\dfrac{1}{s_{X_0}(\bar m)}+\dfrac{\mu}{|X_0|}\)
+
\hat\beta(\lambda)\,
\(\dfrac{1}{s_{X_0}(\hat m)}+\dfrac{\mu}{|X_0|}\)
\end{equation*}
for $0\leq \lambda\leq 1$, 
with $\bar\beta(\lambda)=(1-\lambda)\beta(\lambda)$ and $\hat\beta(\lambda)=\lambda\beta(\lambda)$, $\beta(\lambda)$ defined in \eqref{eq:definition of beta of lambda} (depending on $x_0$), and $m(\lambda)$ from \eqref{eq:parametrization of the x0 segment}.
\label{item:old hypotheses D}
\item 
Given $X_0\in\Gamma_{c_1c_2}$, $\bar Y ,\hat Y\in\Sigma$,  
let $\bar m=\dfrac{\bar Y-X_0}{|\bar Y-X_0|}$ and $\hat m=\dfrac{\hat Y-X_0}{|\hat Y-X_0|}$; $x_0=X_0/|X_0|$.
Let $[\bar Y,\hat Y]_{X_0}$ be the curve defined by
\[
[\bar Y,\hat Y]_{X_0}=\left\{Y(\lambda)=X_0+s_{X_0}(m(\lambda))\,m(\lambda):\lambda\in[0,1]\right\},
\]
where $m(\lambda)$ is the parametrization of $[\bar m,\hat m]_{x_0}$ defined in \eqref{eq:parametrization of the x0 segment}. 
We assume that there exist positive constants $\mu_0$ and $C$ such that for all $X_0\in\Gamma_{c_1c_2}$, $\bar Y ,\hat Y\in\Sigma$, we have 
\[
H^{n-1}\(N_{\mu}\(\left\{[\bar Y,\hat Y]_{X_0}:\frac{1}{4}\leq\lambda\leq\frac{3}{4}\right\}\)\cap\Sigma\)\geq C\,\mu^{n-2}|\bar Y-\hat Y|,
\]
for each $\mu\leq \mu_0$, 
where $H^{n-1}$ denotes the $n-1$ dimensional Hausdorff measure in $\R^{n}$ and $N_{\mu}$ denotes the $\mu$-neighborhood in $\R^{n}$. 
\label{item:old hypotheses E}
\end{enumerate}

{\it Throughout the paper, a structural constant refers to a constant depending only on some or all of the constants in the structural conditions above.}

\begin{remark}\label{rmk:remark on structural conditions}\rm
We begin noticing that from \ref{item:hypotheses A} and \ref{item:old hypotheses C}  we get that $s_X$ is bounded below: 
\begin{equation}\label{eqlower bound for s_X}
s_X(m)\geq c_1/C(\kappa),
\end{equation}
for all $X\in \Gamma_{c_1c_2}$ and $m\in E(X)$.

Also from \ref{item:hypotheses A} and \ref{item:old hypotheses C} we get for $\hat Y=X+s_X(\hat m)\,\hat m$ and $\bar Y=X+s_X(\bar m)\,\bar m$ that
\begin{equation}\label{eq:difference of ms bounded by difference of Ys}
|\hat m-\bar m|\leq 2\min\left\{\frac{1}{|\bar Y-X|},\frac{1}{|\hat Y-X|}\right\}|\bar Y-\hat Y|\leq C|\bar Y-\hat Y|.
\end{equation}
Indeed, from \ref{item:hypotheses A}
\[
\hat m-\bar m=\dfrac{\hat Y-X}{|\hat Y-X|}-\dfrac{\bar Y-X}{|\bar Y-X|}
=
\dfrac{|\bar Y-X|(\hat Y-\bar Y)+(\bar Y-X)(|\bar Y-X|-|\hat Y-X|)}{|\hat Y-X||\bar Y-X|},
\]
so
\[
|\bar m-\hat m|\leq \dfrac{2|\hat Y-\bar Y|}{|\hat Y-X|},\qquad 
|\bar m-\hat m|\leq \dfrac{2|\hat Y-\bar Y|}{|\bar Y-X|}.
\]
Therefore from \ref{item:old hypotheses C} the desired inequality follows since $X\in \Gamma_{c_1c_2}$.

Concerning each of our assumptions we mention the following.
Assumption \ref{item:hypotheses A} guarantees that each ray from $0$ striking $X\in\Gamma_{c_1c_2}$ can be refracted into $\Sigma$ and the refracted ray intersects $\Sigma$ at only one point.
Assumption \ref{item:old hypotheses C} says that $\Gamma_{c_1c_2}$ is sufficiently far from the target $\Sigma$\footnote{The value of the constant $C(\kappa)$ in \ref{item:old hypotheses C} is only needed in Lemma \ref{aux1}.}and it will be applied to show that the ovals used in the definition of  refractor have controlled derivatives.
Assumption \ref{item:old hypotheses D} is crucial to obtain regularity of refractors and is akin to the condition (AW) first introduced in \cite{MaTrudingerWang:regularityofpotentials} and later considered in \cite{loeper:actapaper} and \cite{2010-kim-mccann:DMASM}.
Assumption \ref{item:old hypotheses E} is a form of convexity of $\Sigma$ with respect to points $X\in\Gamma_{c_1c_2}$.

\end{remark}

\begin{remark}\rm
We relate now the structural assumptions introduced with the following assumptions needed to prove existence of refractors \cite[Sect. 5]{gutierrez-huang:nearfieldrefractor}:
\begin{enumerate}[label=\textbf{H.\arabic*}]
\item there exists $\tau$, with $0<\tau<1-\kappa$, such that $x\cdot Y\geq (\kappa+\tau)|Y|$ for all $x\in \Omega$ and $Y\in \Sigma$;\label{eq:hypotheses 1}
\item if $0<r_0<\dfrac{\tau}{1+\kappa}\,\dist(0,\Sigma)$ and $Q_{r_0}=\{t\,x:x\in \Omega, 0<t<r_0\}$, then given $X\in Q_{r_0}$ each ray emanating from $X$ intersects $\Sigma$ in at most one point. \label{eq:hypotheses 2} 
\end{enumerate}
We show that if $\tau$ is sufficiently small, then \ref{eq:hypotheses 1} and \ref{eq:hypotheses 2} 
imply \ref{item:hypotheses A} (a) and \ref{item:old hypotheses C}.
We first claim that 
there are positive constants $C_{\tau,\kappa}$ and $\hat C_{\tau,\kappa}$ such that if
$
\dfrac{X}{|X|}\cdot \dfrac{Y}{|Y|}\geq \kappa +\tau, 
$
and
$
|Y|\geq C_{\tau,\kappa}|X|,
$
for all $Y\in \Sigma$ and $X\in \Gamma_{c_1c_2}$,
then 
\begin{equation*}
\dfrac{Y-X}{|Y-X|}\cdot \dfrac{X}{|X|}\geq\kappa,
\text{ and }
\dfrac{|X|}{|Y-X|}\leq \hat C_{\tau,\kappa},
\end{equation*}
with 
\begin{equation*}
C_{\tau,\kappa}=\dfrac{\sqrt{1-\kappa^{2}}}{(\kappa+\tau)\sqrt{1-\kappa^{2}}-\kappa\sqrt{1-(\kappa+\tau)^{2}}},\quad 
\hat C_{\tau,\kappa}=\dfrac{1}{C_{\tau,\kappa}-1}.
\end{equation*}
Then the desired relation between the assumptions follows noticing that $C_{\tau,\kappa}\rightarrow\infty$ and $\hat C_{\tau,\kappa}\rightarrow 0$ as $\tau\rightarrow 0$.
To prove the claim, fix $X$ and calculate the intersection between the cones $\mathcal C_1=\left\{Y:\dfrac{X}{|X|}\cdot \dfrac{Y}{|Y|}= \kappa +\tau\right\}$, and $\mathcal C_2=\left\{Y:\dfrac{Y-X}{|Y-X|}\cdot \dfrac{X}{|X|}=\kappa\right\}$.
From the sine law, it is easy to see that if $Y$ is in the intersection of these cones, then $|Y|=C_{\tau,\kappa}|X|$. So $|Y|\geq C_{\tau,\kappa}|X|$  and $Y$ is in the interior of $\mathcal C_1$, then $Y$ is in the interior of $\mathcal C_2$,
and $|Y-X|\geq |Y|-|X|\geq \(C_{\tau,\kappa}-1\)|X|$.

\end{remark}

\begin{remark}\rm
When the target $\Sigma$ is $C^2$ one can give a differential condition that is equivalent to \ref{item:old hypotheses D}. 
To do this, we first need to have another parametrization of the curve $[\bar m,\hat m]_{x_0}$.
For $Y\in \Sigma$, recall that from \eqref{eq:tangential gradient}
\[
\nabla^{T}h(x,Y,X_0)=\nabla_x h(x,Y,X_0)-\langle\nabla_x h(x,Y,X_0),x\rangle x,
\]
and since $Y=X_0+s\,m$, for some $m\in E(X_0)$, we have from \eqref{eq:partial h with respect to xi at x0} that
\[
\nabla^{T}h(x_0,Y,X_0)=\kappa|X_0|\,\dfrac{m-\langle m,x_0\rangle x_0}{1-\kappa\langle m,x_0\rangle}
:=v=T_{X_0}(m),\qquad x_0=X_0/|X_0|.
\]
Notice that $v\perp x_0$ and $|v|^{2}\leq \dfrac{\kappa^{2}|X_0|^{2}}{1-\kappa^{2}}$. 
We will write $m$ in terms of $v$, with $\langle m,x_0\rangle\geq\kappa$ and $|m|=1$.
First note that $\langle m,x_0\rangle=\dfrac{|v|^{2}+|X_0|\sqrt{\kappa^{2}|X_0|^{2}-(1-\kappa^{2})|v|^{2}}}{\kappa\(|v|^{2}+|X_0|^{2}\)}$, and thus 
\[
\dfrac{1-\kappa\langle m,x_0\rangle}{\kappa|X_0|}=\dfrac{1-\kappa^{2}}{\kappa}\dfrac{1}{|X_0|+\sqrt{\kappa^{2}|X_0|^{2}-(1-\kappa^{2})|v|^{2}}}:=t(v).
\]
We can then write $m=\langle m,x_0\rangle x_0+t(v)\,v$ and so
\[
m=m(v):=\frac{1}{\kappa}x_0+t(v)(v-X_0).
\]
Given $\bar m,\hat m\in S^{n-1}$ with $\langle\bar m,x_0\rangle\geq\kappa$ and $\langle\hat m,x_0\rangle\geq\kappa$, let
\[
\bar v=\kappa|X_0|\dfrac{\bar m-\langle \bar m,x_0\rangle x_0}{1-\kappa\langle \bar m,x_0\rangle},\qquad 
\text{ and }
\hat v=\kappa|X_0|\dfrac{\hat m-\langle \hat m,x_0\rangle x_0}{1-\kappa\langle \hat m,x_0\rangle}.
\]
Letting $v_{\gamma}=(1-\gamma)\bar v+\gamma\hat v$, we show that the curve $[\bar m,\hat m]_{x_0}$ in \eqref{eq:parametrization of the x0 segment} can be parametrized as follows:
\[
\tilde m(\gamma)=m(v_\gamma)=\frac{1}{\kappa}x_0+t(v_{\gamma})(v_{\gamma}-X_0), \qquad 0<\gamma<1,
\]
that is, $\tilde m(\gamma)=m(\lambda)$ with the change of parameter $\lambda=\dfrac{\gamma t(\bar v)}{(1-\gamma)t(\hat v)+\gamma t(\bar v)}$ (we are abusing the notation $m(\lambda)$ and $m(v)$). 
In fact, from the definition of $\beta(\lambda)$
\begin{equation*}
\bar\beta=(1-\lambda)\beta(\lambda)=\dfrac{t(v_{\gamma})(1-\gamma)}{t(\bar v)}
\qquad 
\text{and}\qquad  
\hat\beta=\lambda\beta(\lambda)=\dfrac{t(v_{\gamma})\gamma}{t(\hat v)};
\end{equation*}
see the end of the proof of Lemma \ref{crucial} for similar calculations with $\beta(\lambda)$.
Also $\bar m=\frac{1}{\kappa}x_0+t(\bar v)(\bar v-X_0)$ and $\hat m=\frac{1}{\kappa}x_0+t(\hat v)(\hat v-X_0)$.
Then,
\begin{align*}
m(\lambda)&=\frac{1}{\kappa}x_0+\beta(\lambda)\((1-\lambda)\bar m+\lambda\hat m-\frac{1}{\kappa}x_0\)\\
&=
\frac{1}{\kappa}x_0+\frac{1}{\kappa}t(v_{\gamma})\dfrac{(1-\gamma)t(\hat v)+\gamma t(\bar v)}{t(\hat v)t(\bar v)}\((1-\lambda)(\kappa\bar m-x_0)+\lambda(\kappa\hat m-x_0)\).
\end{align*}
Since $\kappa\bar m-x_0=\kappa t(\bar v)(\bar v-X_0)$ and $\kappa\hat m-x_0=\kappa t(\hat v)(\hat v-X_0)$, substituting and simplifying yields $m(\lambda)=\tilde m(\gamma)$ as desired.

Therefore, with this reparametrization of the curve $[\bar m, \hat m]_{x_0}$ assumption \ref{item:old hypotheses D} is then equivalent to
\[
\(\dfrac{1}{s_{X_0}(\tilde m(\gamma))}+\dfrac{\mu}{|X_0|}\)\dfrac{1}{t(v_\gamma)}
\geq 
(1-\gamma)\,\(\dfrac{1}{s_{X_0}(\tilde m(0))}+\dfrac{\mu}{|X_0|}\)\dfrac{1}{t(v_0)}
+\gamma\,\(\dfrac{1}{s_{X_0}(\tilde m(1))}+\dfrac{\mu}{|X_0|}\)\dfrac{1}{t(v_1)}
,
\]
for $0<\gamma<1$, with $\tilde m(0)=\bar m$, $\tilde m(1)=\hat m$, $v_0=\bar v$, and $v_1=\hat v$.
In other words, the function 
\[
\Phi(v)=\(\dfrac{1}{s_{X_0}(m(v))}+\dfrac{\mu}{|X_0|}\)\dfrac{1}{t(v)}
\]
is a concave function of $v$ for $ |v|^{2}\leq \dfrac{\kappa^{2}|X_0|^{2}}{1-\kappa^{2}}$ and $v\perp x_0$, i.e., 
concave in a $n-1$-dimensional disk.
Thus, when $\Sigma$ is $C^2$, we obtain that \ref{item:old hypotheses D} is equivalent to
\[
\left.\dfrac{d^{2}}{dt^{2}}\(\Phi(v+t\xi)\)\right|_{t=0}\leq 0
\]
for all $v\perp x_0$ with $ |v|^{2}\leq \dfrac{\kappa^{2}|X_0|^{2}}{1-\kappa^{2}}$ and for all $\xi\perp x_0$.
The domain of $s_{X_0}$ is $E(X_0)$, and the domain of $\Phi(v)$ is $T_{X_0}(E(X_0))$.
The fact that $E(X_0)$ satisfies the convexity assumption that $[\bar m,\hat m]_{x_0}\subset E(X_0)$ for all $\bar m,\hat m\in E(X_0)$ is equivalent that $T_{X_0}(E(X_0))$ is a convex set in the classical sense on the hyperplane perpendicular to $x_0$.

\end{remark}

\subsection{Examples}

We will construct $\Omega\subset S^{n-1}$ and a target $\Sigma$ so that the structural assumptions are satisfied.
Notice that if $\Omega\subset\Omega^{\prime}$, then  $\mathcal C_{\Omega^{\prime}}\subset \mathcal C_{\Omega}$, with $c_1,c_2$ fixed.
To do this construction, we will first choose $\Omega^{\prime}$ and calculate $\mathcal C_{\Omega^{\prime}}$. 
We will then choose a target $\Sigma\subset \mathcal C_{\Omega^{\prime}}$ and next pick $\Omega\subset\Omega^{\prime}$. It will then follow that $\Sigma\subset \mathcal C_{\Omega}$.

Let $\theta=\arccos \kappa $ and $\Omega^{\prime}=\left\{x\in S^{n-1}:x\cdot e_n\geq \cos(\theta/2)\right\}$ where $e_n$ is the unit vector in the vertical direction $x_n$. 

Pick constants $c_1=1$ and $c_2>1$, and let $Y_0=2\,c_2\cos(\theta/2)\,e_n$. 
We claim that  $\mathcal C_{\Omega^{\prime}}=\left\{Y:\dfrac{Y-Y_0}{|Y-Y_0|}\cdot e_n\geq \cos(\theta/2)\right\}:=E$, the cone with vertex at $Y_0$ direction $e_n$ and opening $\theta/2$.
To prove this, let $Y\in E$ and we want to show that $Y\in C_X$ for all $X\in \Gamma_{c_1c_2}$ (defined with $\Omega^{\prime}$).
That $Y\in E$ means $\angle\(Y-Y_0,e_n\)\leq \theta/2$, where $\angle$ denotes the angle between the vectors. 
Obviously, $Y\in  C_{X}$ if and only if $\angle (Y-X,X)\leq\theta$. From the choice of $Y_0$, it is easy to see that $\angle (Y_0-X,X)\leq\theta $ for all $X\in \Gamma_{c_1c_2}$, i.e., $Y_0\in C_X$.

We have $Y=Y_0+v$ with $\angle(v,e_n)\leq \theta/2$. Let $\bar Y=X+v$. Since $\angle(\bar Y-X,X)=\angle(v,X)\leq\angle(v,e_n)+\angle(e_n,X)\leq\frac{\theta}{2}+\frac{\theta}{2}$, it follows that $\bar Y\in C_{X}$. 
Then from the convexity of $C_{X}$ we obtain $\dfrac{Y_0+\bar Y}{2}\in C_{X}$. 
Since $Y=\bar Y+Y_0-X$, it follows that $\angle (Y-X,X)=\angle (\bar Y+Y_0-2X,X)=
\angle\(\dfrac{\bar Y+Y_0}{2}-X,X\)\leq \theta$. So $Y\in C_X$ and the claim is proved.

Now, we choose $\Sigma$ the planar disk centered at $0$ with radius $R$ at height $M$, that is,
 \[
 \Sigma=\left\{Y=(Y^{\prime},Y_n):|Y^{\prime}|\leq R\;Y_n=M\right\}.
 \]
If we pick 
$M=C+2\,c_2\,\cos(\theta/2)=C+2\,c_2\,\sqrt{\frac{1+\kappa}{2}}$ with $C$ any positive constant and pick $R\leq \sqrt{\dfrac{1-\kappa}{1+\kappa}}C$, then
it is easy to verify that $\Sigma\subset E=\mathcal C_{\Omega^{\prime}}$.

Next, we will choose $\Omega\subset \Omega^{\prime}$ so that if $\bar Y,\hat Y\in\Sigma$, then $[\bar m,\hat m]_{x}\subseteq E(X)$, for all $X\in \Gamma_{c_1c_2}$, where $\bar m=\dfrac{\bar Y-X}{|\bar Y-X|}$,
$\hat m=\dfrac{\hat Y-X}{|\hat Y-X|}$ and $E(X)$ is the set of visibility directions in $S^{n-1}$ of the target $\Sigma$ from the point $X$.
First notice that since $\cos(\theta/2)=\sqrt{\dfrac{1+\kappa}{2}}$ and $\sin(\theta/2)=\sqrt{\dfrac{1-\kappa}{2}}$, we have $x\in\Omega^{\prime}$ if and only if $\dfrac{|x^{\prime}|}{x_n}\leq \sqrt{\dfrac{1-\kappa}{1+\kappa}}$, with $x_n>0$.
Now define 
\[
\Omega=\left\{x\in S^{n-1}:x_n>0,\,\dfrac{|x^{\prime}|}{x_n}\leq \dfrac{R}{M}\right\}.
\] 
Since $\dfrac{R}{M}\leq \dfrac{\sqrt{\dfrac{1-\kappa}{1+\kappa}}C}{C+2c_2\sqrt{\dfrac{1-\kappa}{1+\kappa}}}$, we obtain that $\Omega\subset \Omega^{\prime}$.

If $X\in\Gamma_{c_1c_2}$ and $\bar Y,\hat Y\in\Sigma$, then we show $[\bar m,\hat m]_{x}\subseteq E(X)$,
where $\bar m=\dfrac{\bar Y-X}{|\bar Y-X|}$ and $\hat m=\dfrac{\hat Y-x}{|\hat Y-X|}$ and $x=\dfrac{X}{|X|}$. Since $\bar Y,\hat Y\in \mathcal C_{\Omega}$, we have $\bar m\cdot x\geq\kappa$ and $\hat m\cdot x\geq\kappa$.
We have from \eqref{eq:writing of m in the curve hatm bar m} that $m=\frac{1}{\kappa}x+\bar\beta(\bar m-\frac{1}{\kappa}x)+\hat\beta(\hat m-\frac{1}{\kappa}x)$ for $m\in [\bar m,\hat m]_{x}$, and we need to show that the ray $X+s\,m$ strikes $\Sigma$ for some $s$ (that is, $s=s_X(m)$).
If $s=\dfrac{M-X_n}{m_n}$, then will show that $Y=X+\dfrac{M-X_n}{m_n}m\in\Sigma$.
Indeed, write $Y=(Y',Y_n)$. Clearly $Y_n=M$.
If $|\bar Y^{\prime}|,|\hat Y^{\prime}|\leq R$, will prove that $|Y^{\prime}|\leq R$.
We have with $m=(m',m_n)$ that
\[
Y^{\prime}=X^{\prime}+(M-X_n)\dfrac{m^{\prime}}{m_n}=X^{\prime}+(M-X_n)\dfrac{\frac{1}{\kappa}x^{\prime}\(1-\bar\beta-\hat\beta\)+\bar\beta\bar m^{\prime}+\hat\beta\hat m^{\prime}}{m_n},
\]
and 
\begin{align*}
Y^{\prime}&=\dfrac{\bar\beta \bar m_n}{m_n}\(X^{\prime}+\(M-X_n\)\dfrac{\bar m^{\prime}}{\bar m_n}\)+
\dfrac{\hat\beta \hat m_n}{m_n}\(X^{\prime}+\(M-X_n\)\dfrac{\hat m^{\prime}}{\hat m_n}\)\\
&\qquad +
\(1-\dfrac{\bar\beta\bar m_n+\hat\beta\hat m_n}{m_n}\)X^{\prime} +\dfrac{M-X_n}{m_n}\(\frac{1}{\kappa}x^{\prime}\(1-\bar\beta-\hat\beta\)\).
\end{align*}
Combining the last two terms and simplifying yields
\[
Y^{\prime}
=\dfrac{\bar\beta \bar m_n}{m_n}\bar Y^{\prime}
+
\dfrac{\hat\beta \hat m_n}{m_n}\hat Y^{\prime}+\frac{1}{\kappa}x^{\prime}\dfrac{1-\bar\beta-\hat\beta}{m_n}M.
\]
Therefore, 
$|Y^{\prime}|\leq R\dfrac{\bar\beta\bar m_n+\hat\beta\hat m_n}{m_n}+\dfrac{1}{\kappa}|x^{\prime}|\,
\dfrac{1-\bar\beta-\hat\beta}{m_n}\,M\leq R$, 
where we have used that $|x^{\prime}|\leq \dfrac{x_n R}{M}$ since $x\in \Omega$.
Thus, $[\bar m,\hat m]_x\subseteq E(X)$.

In addition,  
\begin{align*}
\dfrac{1}{s}=\dfrac{m_n}{M-X_n}
&=\dfrac{\frac{1}{\kappa}x_n(1-(\bar\beta+\hat\beta))+\bar\beta\bar m_n+\hat\beta\hat m_n}{M-X_n}\\
&\geq
\dfrac{\bar\beta\bar m_n}{M-X_n}+\dfrac{\hat\beta\hat m_n}{M-X_n}=\dfrac{\bar\beta}{\bar s}+\dfrac{\hat\beta}{\hat s}.
\end{align*}
and so the concavity assumption in \ref{item:old hypotheses D} holds with $\mu=0$.

Therefore the example described satisfies the assumptions \ref{item:hypotheses A}, and \ref{item:old hypotheses D}. In order to satisfy \ref{item:old hypotheses C}, it is enough to keep $c_2$ fixed and pick $C$ large enough.
It remains to verify that example satisfies \ref{item:old hypotheses E}.
For this we use the following lemma.

\begin{lemma}
Let $\gamma:[a,b]\rightarrow R^{n}$ be a smooth curve such that $|\gamma^{\prime}(t)|=1$ and $|\gamma^{\prime\prime}(t)|\leq M_1$ for all $t\in [a,b]$. In addition, assume $M_2|t_1-t_2|\leq|\gamma(t_1)-\gamma(t_2)|$ for all $t_1,t_2\in[a,b]$. Let $T_t$ denote the hyperplane passing through $\gamma(t)$ with normal $\gamma^{\prime}(t)$ and let $D_{\mu}(t)=B_\mu(\gamma(t))\cap T_t$, and 
$N_{\mu}=\bigcup_{t\in [a,b]}D_{\mu}(t)$.
Then, there exists $\mu_0$  and $C$ depending only on $M_1,M_2$ such that for $\mu\leq\mu_0$, we have $H^{n}(N_{\mu})\geq C\mu^{n-1}|\gamma(b)-\gamma(a)|$.
\end{lemma}
\begin{proof}
First observe that there exists $\mu_0$ such that if $\mu\leq\mu_0$, then $D_{\mu}(t_1)\cap D_{\mu}(t_2)=\emptyset$, for $t_1\neq t_2$.

Consider the cylinder in $\R^{n}$ given by $D\times [a,b]=\left\{(x^{\prime},t):|x^{\prime}|\leq\mu;\;t\in[a,b]\right\}$, where $D=\{(x',0):|x'|\leq \mu\}$, and define $F:D\times [a,b]\rightarrow N_{\mu}$ by
\[
F(x^{\prime},t)=\gamma(t)+A(t)(x^{\prime},0)
\]
where $A(t)$ is the $n\times n$ matrix whose column vectors are $\{\eta_1(t),...,\eta_{n-1}(t),\gamma^{\prime}(t)\}$ where $\eta_{i}(t)$ are chosen so that they are smooth with $A(t)A^{T}(t)=I$; here $(x^{\prime},0)$ is a column vector.
Notice that $F$ is one to one and each disk $D\times\{t\}$ is mapped to $D_{\mu}(t)$.
By the formula of change of variables
\begin{align*}
H^{n}\(N_\mu\)&=H^{n}\(F(D\times [a,b])\)=\int_{D\times [a,b]}|\det DF(x^{\prime},t)|dx^{\prime}dt\\
&\geq
CH^{n}(D\times [a,b])\geq C(b-a)\mu^{n-1}\geq C|\gamma(b)-\gamma(a)|\mu^{n-1},
\end{align*}
provided that $|\det DF(x^{\prime},t)|\geq C$ for some $C>0$. Indeed,
note that the matrix $DF(x^{\prime},t)$ has column vectors given by $\eta_1(t),...,\eta_{n-1}(t)$ and its last column vector is $\gamma^{\prime}(t)+x_1\eta_1^{\prime}(t)+...+x_{n-1}\eta_{n-1}^{\prime}(t)$, $x'=(x_1,\cdots ,x_{n-1})$.
Therefore, we can expand $\det DF(x^{\prime},t)=\det A(t)+\sum_{k=1}^{n-1}x_{k}\det \Lambda_{k}(t)$, where $\Lambda_{k}(t)$ is the matrix whose column vectors are $\eta_1(t),...,\eta_{n-1}(t),\eta_{k}^{\prime}(t)$.
It follows that $|\det DF(x^{\prime},t)|\geq |\det A(t)|-\sum_{k=1}^{n-1}|x_k||\det \Lambda_{k}(t)|\geq 1-\mu_0\sum_{k=1}^{n-1}|\det \Lambda_{k}(t)|\geq 1-\mu_0\, C$,
with $C$ depending only on $M_1,M_2$ and $n$. Then choosing $\mu_0$ sufficiently small the lemma follows.

\end{proof}
Finally, to verify that our example satisfies \ref{item:old hypotheses E}, we notice that the curves $[\bar Y,\hat Y]_{X_0}$ in the example satisfy the assumptions of the last lemma (the curves can be reparametrized to have $|\gamma^{\prime}|=1$) in $\R^{n-1}$ so it is applicable to our case, obtaining constants that depend only on the structure.
Also, varying the parameters $c_2$, $C$ and $R$ in the construction, we obtain a family of examples.

\section{Preliminary results for ovals and a maximum principle}\label{sec:preliminary for ovals and max principle}
\setcounter{equation}{0}

We analyze the function $h(x,Y,X_0)$ for $Y\in\Sigma$ and $X_0\in\Gamma_{c_1c_2}$; $x\in S^{n-1}$, corresponding to the oval $\mathcal O(Y,b)$.
From \ref{item:hypotheses A}, we can write $Y=X_0+s\,m$ with $m\in S^{n-1}$, $s=s_{X_0}(m)>0$, $x_0\cdot m\geq \kappa$, and recall that $b=|X_0|+\kappa|Y-X_0|=|X_0|+\kappa\, s$; $x_0=X_0/|X_0|$. Hence 
\begin{align}
b-\kappa^{2}\,x\cdot Y&=|X_0|\(1-\kappa^{2}\,x\cdot x_0\)+\kappa \,s\,\(1-\kappa\, x\cdot m\)\label{eq:expression for b -kappa2xcdotY}\\
b^{2}-\kappa^{2}|Y|^{2}&=(1-\kappa^{2})|X_0|^{2}+2\kappa \,s\,|X_0|\,(1-\kappa x_0\cdot m).\label{eq:expression for b2 -kappa2Y2}\end{align}
Settting 
\begin{equation}\label{eq:definition of B and C}
B=\dfrac{b-\kappa^{2}\, x\cdot Y}{1-\kappa^{2}}\quad \text{and}\quad C=\dfrac{b^{2}-\kappa^{2}|Y|^{2}}{1-\kappa^{2}},
\end{equation} 
we then can write 
\[
h(x,Y,X_0)=B-\sqrt{B^{2}-C}.
\]

In order to get to our crucial Lemma \ref{crucial}, first we need to prove three auxiliary lemmas.

\begin{lemma}\label{aux1}
Assume \ref{item:hypotheses A} and \ref{item:old hypotheses C}.
Let $\bar Y,\hat Y\in\Sigma$, and consider the ovals $\mathcal O(\bar Y,\bar b), \mathcal O(\hat Y,\hat b)$. If $X_0\in\Gamma_{c_1c_2}$ is a common point to both ovals, then 
with the notation above we have 
\[
\hat B\geq \bar B-\sqrt{\bar B^{2}-\bar C}, \qquad \forall x\in S^{n-1}.
\]
\end{lemma}
\begin{proof}
Since $h(x,\bar Y,X_0)>0$, it follows that 
$\bar B-\sqrt{\bar B^{2}-\bar C}
=
\dfrac{\bar C}{\bar B+\sqrt{\bar B^{2}-\bar C}}
\leq\dfrac{\bar C}{\bar B}$.
So, it is enough to show $\bar C\leq \bar B\,\hat B$, which is equivalent to show that
\begin{align*}
&\dfrac{(1-\kappa^{2})|X_0|^{2}+2\kappa \,\bar s\,|X_0|\,(1-\kappa\, x_0\cdot \bar m)}{1-\kappa^{2}}\\
&\leq
\dfrac{\left(|X_0|\(1-\kappa^{2}\,x\cdot x_0\)+\kappa\, \bar s\,\(1-\kappa\,x\cdot \bar m\)\right)\left(|X_0|(1-\kappa^{2}\,x\cdot x_0)+\kappa \,\hat s\,(1-\kappa\,x\cdot \hat m)\right)}{(1-\kappa^{2})^{2}},
\end{align*}
where we have used the notation $\bar Y=X_0+\bar s\,\bar m$ and $\hat Y=X_0+\hat s\,\hat m$, $\bar s=s_{X_0}(\bar m),\hat s=s_{X_0}(\hat m)>0$, $x_0=X_0/|X_0|$ with $x_0\cdot \bar m\geq \kappa$, $x_0\cdot \hat m\geq \kappa$.
The last inequality is equivalent to
\begin{align*}
&|X_0|^{2}\left(1-\frac{(1-\kappa^{2}\, x\cdot x_0)^{2}}{(1-\kappa^{2})^{2}}\right)+\dfrac{2\kappa|X_0|\,\bar s(1-\kappa\, x_0\cdot \bar m)}{1-\kappa^{2}}\\
&\leq
\dfrac{\kappa|X_0|(1-\kappa^{2}\, x\cdot x_0)}{(1-\kappa^{2})^{2}}\(\bar s\,(1-\kappa\,x\cdot \bar m)+\hat s\,(1-\kappa\, x\cdot \hat m)\)\\
&\qquad +\dfrac{\kappa^{2}}{(1-\kappa^{2})^{2}}\,\bar s\,\hat s\,(1- \kappa\,x\cdot \bar m)(1-\kappa\,x\cdot \hat m).
\end{align*}
The left hand side of the last inequality is $\leq |X_0|^{2}+2\kappa|X_0|\bar s$ and the right hand side is $\geq \dfrac{\kappa^{2}}{(1-\kappa^{2})^{2}}\,\bar s\,\hat s\,(1-\kappa)^{2}=\dfrac{\kappa^{2}}{(1+\kappa)^{2}}\,\bar s\,\hat s$. 
Therefore, if $|X_0|^{2}+2\kappa|X_0|\bar s\leq\dfrac{\kappa^{2}}{(1+\kappa)^{2}}\,\bar s\,\hat s$, then the desired inequality follows.
This is equivalent to
\[
\frac{|X_0|}{\bar s}\frac{|X_0|}{\hat s}+2\kappa\frac{|X_0|}{\hat s}\leq \frac{\kappa^{2}}{(1+\kappa)^{2}}
\]
which follows from \ref{item:old hypotheses C}.

\end{proof}

A second auxiliary calculus lemma is as follows.
\begin{lemma}\label{aux2}
Consider the two variable function $f(B,C)=B-\sqrt{B^{2}-C}$ on the set $0\leq C\leq B^{2}$ and $B\geq 0$.
Fix $(\bar B,\bar C)$ in that set, and suppose that $f(\bar B,\bar C)\leq B$.
Then $f(B,C)\leq f(\bar B,\bar C)$ if and only if $C-\bar C\leq 2(B-\bar B)f(\bar B,\bar C)$.
In addition, if $C-\bar C\leq 2(B-\bar B)f(\bar B,\bar C)-E$ for some $E\geq 0$, then $f(B,C)\leq f(\bar B,\bar C)-\dfrac{E}{B+\sqrt{B^{2}-C}-f(\bar B,\bar C)}$.
\end{lemma}
\begin{proof}
Assume that $C-\bar C\leq 2(B-\bar B)f(\bar B,\bar C)-E$, for some $E\geq 0$.
Then 
\begin{align*}
f(B,C)-f(\bar B,\bar C)&=\dfrac{C-\bar C-(f(B,C)+f(\bar B,\bar C))(B-\bar B)}{\sqrt{B^{2}-C}+\sqrt{\bar B^{2}-\bar C}}\\
&\leq
\dfrac{2(B-\bar B)f(\bar B,\bar C)-E-(f(B,C) +f(\bar B,\bar C))(B-\bar B)}{\sqrt{B^{2}-C}+\sqrt{\bar B^{2}-\bar C}}\\
&=
\dfrac{(f(\bar B,\bar C) -f(B,C))(B-\bar B)-E}{\sqrt{B^{2}-C}+\sqrt{\bar B^{2}-\bar C}}.
\end{align*}
Therefore,
\[
\(f(B,C)-f(\bar B,\bar C)\)\(1+\frac{B-\bar B}{\sqrt{B^{2}-C}+\sqrt{\bar B^{2}-\bar C}}\)\leq \dfrac{-E}{\sqrt{B^{2}-C}+\sqrt{\bar B^{2}-\bar C}}
\]
which implies 
\[
f(B,C)\leq f(\bar B,\bar C)-\dfrac{E}{B+\sqrt{B^{2}-C}-f(\bar B,\bar C)}.
\]
Conversely,  assume $f(B,C)\leq f(\bar B,\bar C)$, that is, $B-\sqrt{B^{2}-C}\leq f(\bar B,\bar C)$ which implies
\[
0\leq B-f(\bar B,\bar C)\leq\sqrt{B^{2}-C},
\]
where the first inequality is from the assumption. 
Hence,
\[
C\leq 2Bf(\bar B,\bar C)-f(\bar B,\bar C)^{2}=2(B-\bar B)f(\bar B,\bar C)+\bar C.
\]
\end{proof}

The third auxiliary lemma says that the oval passing thru $X_0$ is enclosed by the ellipsoid with axis $m$ and eccentricity $\kappa$ passing thru $X_0$ when $x_0\cdot m\geq \kappa$.

\begin{lemma}\label{aux3}
Suppose $x_0\cdot m\geq\kappa$ and let $Y=X_0+s\,m$ with $s>0$; $x_0=X_0/|X_0|$.
Then 
\[
\left\{X:|X|+\kappa|X-Y|\leq |X_0|+\kappa|X_0-Y|\right\}\subseteq
\left\{X:|X|-\kappa\,X\cdot m\leq |X_0|-\kappa\, X_0\cdot m\right\}.
\]
In particular, 
\[
h(x,Y,X_0)\leq \dfrac{|X_0|(1-\kappa\, x_0\cdot m)}{1-\kappa\, x\cdot m} 
\]
for all $x\in S^{n-1}$.
\end{lemma}
\begin{proof}
Let $X$ with $|X|+\kappa|X-Y|\leq |X_0|+\kappa|X_0-Y|$. Then
\begin{align*}
|X|-\kappa\, X\cdot m&=|X|+\kappa|X-Y|-\kappa\, X\cdot m-\kappa|X-Y|\leq|X_0|+\kappa|X_0-Y|-\kappa\,\(X\cdot m+|X-Y|\)\\
&=
|X_0|+\kappa|X_0-Y|-\kappa\,\((X-Y)\cdot m+|X-Y|\)-\kappa\, Y\cdot m\\
&\leq
|X_0|+\kappa|X_0-Y|-\kappa\, Y\cdot m=|X_0|+\kappa|X_0-Y|-\kappa\, (Y-X_0)\cdot m-
\kappa\, X_0\cdot m\\
&=
|X_0|+\kappa \,s-\kappa\, s\,m\cdot m-\kappa\, X_0\cdot m
=
|X_0|-\kappa\, X_0\cdot m.
\end{align*}
\end{proof}

We are now ready to prove a crucial lemma akin to \cite[Prop. 5.1]{loeper:actapaper} and \cite[Thm. 4.10 (DMASM)]{2010-kim-mccann:DMASM} in optimal mass transport.

\begin{lemma}\label{crucial}
Assume \ref{item:hypotheses A}, \ref{item:old hypotheses C}, and \ref{item:old hypotheses D}.
There exists a structural constant $C_0>0$ such that if
 $\bar Y,\hat Y \in\Sigma$,  $X_0\in\Gamma_{c_1c_2}$ with 
 $\bar Y=X_0+\bar s\,\bar m$, $\bar s=s_{X_0}(\bar m)$, $\hat Y=X_0+\hat s\,\hat m$, $\hat s=s_{X_0}(\hat m)$,
then 
\[
C_0\,\lambda\,(1-\lambda)\,|\bar Y-\hat Y|^{2}\,|x-x_0|^{2}+h(x,Y,X_0)\leq\max\{h(x,\bar Y,X_0),h(x,\hat Y,X_0)\}
\]
for all $x\in S^{n-1}$, $Y=X_0+s_{X_0}(m(\lambda))$, $m(\lambda)\in [\bar m,\hat m]_{x_0}$ and $0<\lambda<1$.

\end{lemma}

\begin{proof}
Fix $x\in S^{n-1}$ and assume without lost of generality that $h(x,\bar Y,X_0)\geq h(x,\hat Y,X_0)$, that is,
$f(\hat B,\hat C)\leq f(\bar B,\bar C)$.
We will show that
\[
C_0\lambda(1-\lambda)|\bar Y-\hat Y|^{2}|x-x_0|^{2}+h(x,Y,X_0)\leq h(x,\bar Y,X_0).
\]
By Lemma \ref{aux1}, we have 
$\hat B\geq \bar B-\sqrt{\bar B^{2}-\bar C}=f(\bar B,\bar C)$ so we can apply Lemma \ref{aux2} to obtain 
\[
\hat C-\bar C\leq 2f(\bar B,\bar C)(\hat B-\bar B).
\]
This means
\[
2\kappa|X_0|\(\hat s\,(1-\kappa\, x_0\cdot \hat m)-\bar s\,(1-\kappa\, x_0\cdot \bar m)\)
\leq
2\kappa f(\bar B,\bar C)\(\hat s\,(1-\kappa\, x\cdot \hat m)-\bar s\,(1-\kappa\, x\cdot \bar m)\)
\]
which is equivalently to
\begin{equation}\label{eq:hat less than or equal than bar}
\hat s\,\(|X_0|(1-\kappa\, x_0\cdot \hat m)-f(\bar B,\bar C)\,(1-\kappa\, x\cdot \hat m)\)
\leq
\bar s\,\(|X_0|(1-\kappa\, x_0\cdot \bar m)-f(\bar B,\bar C)\,(1-\kappa\, x\cdot \bar m)\).
\end{equation}
We will show that
\begin{equation}\label{eq:inequality with E first}
C-\bar C\leq 2f(\bar B,\bar C)(B-\bar B)-E
\end{equation}
with some $E$ to be chosen at the end, where $B$ and $C$ are given in \eqref{eq:definition of B and C} corresponding to 
$Y=X_0+s_{X_0}(m(\lambda))$. 
To show \eqref{eq:inequality with E first}, is equivalent to show that
\[
|X_0|\,\(s\,(1-\kappa\, x_0\cdot m)-\bar s\,(1-\kappa\, x_0\cdot \bar m)\)\leq
f(\bar B,\bar C)\,\(s\,(1-\kappa\, x\cdot m)-\bar s\,(1-\kappa\, x\cdot \bar m)\)-\frac{(1-\kappa^{2})E}{2\kappa},
\]
for $m=m(\lambda)\in [\bar m,\hat m]_{x_0}$, $s=s_{X_0}(m(\lambda))$, and $0<\lambda<1$.
Equivalently, we will show 
\begin{align}\label{eq:inequality with E}
 &s\,\(|X_0|(1-\kappa\, x_0\cdot m)-f(\bar B,\bar C)\,(1-\kappa\, x\cdot m)\)\\
 &\leq
\bar s\,\(|X_0|(1-\kappa\, x_0\cdot \bar m)-f(\bar B,\bar C)\,(1-\kappa\, x\cdot \bar m)\)-\frac{(1-\kappa^{2})E}{2\kappa}.\notag
\end{align}
Indeed, first recall that $m$ can be written as in \eqref{eq:writing of m in the curve hatm bar m} with $\bar\beta(\lambda)=(1-\lambda)\beta(\lambda)$ and $\hat\beta(\lambda)=\lambda\beta(\lambda)$ with $\beta(\lambda)$ defined in \eqref{eq:definition of beta of lambda}.
From \eqref{eq:writing of m in the curve hatm bar m} 
\begin{align*}
&s\,\(|X_0|(1-\kappa\, x_0\cdot m)-f(\bar B,\bar C)\,(1-\kappa\, x\cdot m)\)\\
&=
s\,\(|X_0|\(\bar\beta\,(1-\kappa\, x_0\cdot \bar m)+\hat\beta\,(1-\kappa\, x_0\cdot \hat m)\)-f(\bar B,\bar C)(1-\kappa\, x\cdot m)\)\\
&=
s\,\bar\beta\,\(|X_0|(1-\kappa\, x_0\cdot \bar m)-f(\bar B,\bar C)\,(1-\kappa\, x\cdot \bar m)\)\\
&\qquad +
s\,\hat\beta\,\(|X_0|(1-\kappa\, x_0\cdot \hat m)-f(\bar B,\bar C)\,(1-\kappa\, x\cdot \hat m)\)\\
&\qquad \qquad +
s\,f(\bar B,\bar C)\,\(\bar\beta\,(1-\kappa\, x\cdot \bar m)+\hat\beta\,(1-\kappa\, x\cdot \hat m)-(1-\kappa\, x\cdot m)\)\\
&=I+II+III.
\end{align*}
Again from \eqref{eq:writing of m in the curve hatm bar m} 
\[
\bar\beta\,(1-\kappa\, x\cdot \bar m)+\hat\beta\,(1-\kappa\, x\cdot \hat m)-(1-\kappa\, x\cdot m)=
\(\bar\beta+\hat\beta-1\)(1- x\cdot x_0)=\frac{1}{2}\(\bar\beta+\hat\beta-1\)|x-x_0|^{2}.
\]
From \eqref{eq:hat less than or equal than bar}
\begin{align*}
II&=s\,\hat\beta\,\(|X_0|(1-\kappa\, x_0\cdot \hat m)-f(\bar B,\bar C)\,(1-\kappa\, x\cdot \hat m)\)\\
&\leq
\dfrac{s\,\hat\beta\,\bar s}{\hat s}\,\(|X_0|(1-\kappa\, x_0\cdot \bar m)-f(\bar B,\bar C)\,(1-\kappa\, x\cdot \bar m)\).
\end{align*}

If we now let
\[
K:=\dfrac{\mu}{|X_0|}\,(1-\beta(\lambda)),
\]
$\bar \beta(\lambda)+\hat \beta(\lambda)=\beta(\lambda)$,
then with simplified notation \ref{item:old hypotheses D} reads
\begin{equation}\label{nn}
\bar\beta\,\hat s+\hat\beta\,\bar s\leq \dfrac{\bar s\,\hat s}{s}+K\,\bar s\,\hat s.
\end{equation}
We also notice that since $f(\bar B,\bar C)=h(x,\bar Y,X_0)$ and $x_0\cdot \bar m\geq \kappa$, by Lemma \ref{aux3} 
\[
|X_0|(1-\kappa\, x_0\cdot \bar m)-f(\bar B,\bar C)\,(1-\kappa\, x\cdot \bar m)\geq 0.
\]
Therefore
\begin{align}\label{keybis}
&s\,\(|X_0|(1-\kappa\, x_0\cdot m)-f(\bar B,\bar C)\,(1-\kappa\, x\cdot m)\)=I+II+III\notag \\
&\leq
\dfrac{s\,(\bar\beta\hat s+\hat\beta\bar s)}{\hat s}\,\(|X_0|(1-\kappa\, x_0\cdot \bar m)-f(\bar B,\bar C)\,(1-\kappa\, x\cdot \bar m)\)-s\,f(\bar B,\bar C)\,\frac{(1-(\bar\beta+\hat\beta))}{2}\,|x-x_0|^{2}\notag\\
&\leq
\bar s\,\(|X_0|(1-\kappa\, x_0\cdot \bar m)-f(\bar B,\bar C)\,(1-\kappa\, x\cdot \bar m)\)
+
K\,\bar s\,s\,\(|X_0|(1-\kappa\, x_0\cdot \bar m)-f(\bar B,\bar C)\,(1-\kappa\, x\cdot \bar m)\)\notag\\
&\qquad -s\,f(\bar B,\bar C)\,\frac{(1-(\bar\beta+\hat\beta))}{2}\,|x-x_0|^{2}.
\end{align}
To estimate the middle term in the last inequality we shall prove that for some $\delta>0$\begin{equation}\label{i}
K\,\bar s\,\big(|X_0|(1-\kappa\langle x_0,\bar m\rangle-h(x,\bar Y,X_0)(1-\kappa\langle x,\bar m\rangle)\big)
\leq 
\delta\,(1-\beta(\lambda))\,|x-x_0|^{2}\,h(x,\bar Y,X_0),
\end{equation}
where $h(x,\bar Y,X_0)=f(\bar B,\bar C)$.
From the definition of $K$ this inequality is equivalent to
\begin{equation}\label{iii}
\bar s\,\big(|X_0|(1-\kappa\langle x_0,\bar m\rangle-h(x,\bar Y,X_0)(1-\kappa\langle x,\bar m\rangle)\big)
\leq 
\dfrac{\delta}{\mu}\,|X_0|\,|x-x_0|^{2}\,h(x,\bar Y,X_0).
\end{equation}
Let 
\[
\Delta=|X_0|(1-\kappa\langle x_0,\bar m\rangle-h(x,\bar Y,X_0)(1-\kappa\langle x,\bar m\rangle).
\]
Writing $X=h(x,\bar Y,X_0)x$ with $\bar Y=X_0+\bar s\,\bar m$,
we have $|X|+\kappa|X-\bar Y|=|X_0|+\kappa|X_0-\bar Y|$, which after simplification implies that
\[
\Delta=\dfrac{\kappa^{2}|X-X_0|^{2}-(|X|-|X_0|)^{2}}{2\,\kappa\,\bar s}.
\]
By calculation, the right hand side of the last identity is equal to 
\[
\dfrac{|X||X_0||x-x_0|^{2}-(1-\kappa^{2})|X-X_0|^{2}}{2\,\kappa\,\bar s}\leq \dfrac{|X||X_0||x-x_0|^{2}}{2\,\kappa\,\bar s}=
\dfrac{h(x,\bar Y,X_0)|X_0||x-x_0|^{2}}{2\,\kappa\,\bar s}
\]
implying \eqref{i} with $\delta=\mu/(2\kappa)$.
Therefore inserting \eqref{i} in \eqref{keybis} yields
\begin{align*}
&s\,\(|X_0|(1-\kappa\, x_0\cdot m)-f(\bar B,\bar C)\,(1-\kappa\, x\cdot m)\)\\
&\leq
\bar s\,\Delta
-\dfrac12\,\(1-\dfrac{\mu}{\kappa}\)\,s\,f(\bar B,\bar C)\,\(1-\beta(\lambda)\)\,|x-x_0|^{2}.
\end{align*}
Therefore we have proved \eqref{eq:inequality with E} with 
\[
E=\(1-\dfrac{\mu}{\kappa}\)\,\dfrac{\kappa \,s\,f(\bar B,\bar C)\,(1-(\bar\beta+\hat\beta))\,|x-x_0|^{2}}{1-\kappa^{2}}.
\]
and consequently \eqref{eq:inequality with E first}.

Since $X_0$ is on both ovals $\mathcal O(Y,b),\mathcal O(\bar Y,\bar b)$, then by Lemma \ref{aux1}, $B\geq f(\bar B,\bar C)$ . So from \eqref{eq:inequality with E first} we can apply the last part of Lemma \ref{aux2} to get
\[
f(B,C)+\dfrac{E}{B+\sqrt{B^{2}-C}-f(\bar B,\bar C)}\leq f(\bar B,\bar C),
\]
that is,
\[
h(x,Y,X_0)+\dfrac{E}{B+\sqrt{B^{2}-C}-f(\bar B,\bar C)}\leq h(x,\bar Y,X_0).
\]

Finally,
to complete the proof of the lemma, we estimate $\dfrac{E}{B+\sqrt{B^{2}-C}-f(\bar B,\bar C)}$ from below.
We shall first prove that $1-(\bar\beta+\hat\beta)\geq C_{\kappa}\,\lambda\,(1-\lambda)\,|\bar m-\hat m|^{2}$.
In fact, 
\begin{align*}
1-(\bar\beta+\hat\beta)=1-\beta(\lambda)&=\dfrac{\kappa|\xi|^{2}+\langle x_0,\xi\rangle+\sqrt{\langle x_0,\xi\rangle^{2}-(1-\kappa^{2})|\xi|^{2}}}{\kappa|\xi|^{2}}\\
&=
\dfrac{\(\kappa|\xi|^{2}+\langle x_0,\xi\rangle\)^{2}-\(\langle x_0,\xi\rangle^{2}-(1-\kappa^{2})|\xi|^{2}\)}{\kappa|\xi|^{2}\,\(\kappa|\xi|^{2}+\langle x_0,\xi\rangle-\sqrt{\langle x_0,\xi\rangle^{2}-(1-\kappa^{2})|\xi|^{2}}\)}.
\end{align*}
We have $(\kappa|\xi|^{2}+\langle x_0,\xi\rangle)^{2}-(\langle x_0,\xi\rangle^{2}-(1-\kappa^{2})|\xi|^{2})=|\xi|^{2}(\kappa^{2}|\xi|^{2}+2\kappa\langle x_0,\xi\rangle+1-\kappa^{2})=|\xi|^{2}(|\kappa\xi+x_0|^{2}-\kappa^{2})=|\xi|^{2}\kappa^{2}(|m_{\lambda}|^{2}-1)$.
Therefore 
\[
1-\beta(\lambda)=\dfrac{\kappa(1-|m_{\lambda}|^{2})}{-\kappa|\xi|^{2}-\langle x_0,\xi\rangle+\sqrt{\langle x_0,\xi\rangle^{2}-(1-\kappa^{2})|\xi|^{2}}}.
\]
Since $1-\beta(\lambda)> 0$ and $|m_{\lambda}|< 1$, for $0<\lambda <1$, it follows that $\Delta:=-\kappa|\xi|^{2}-\langle x_0,\xi\rangle+\sqrt{\langle x_0,\xi\rangle^{2}-(1-\kappa^{2})|\xi|^{2}})>0$ and since $|\xi|\leq 1+(1/\kappa)$, 
$\Delta$ is bounded above by a constant depending only on $\kappa$.
Since $1-|m_{\lambda}|^{2}=\lambda\,(1-\lambda)\,|\bar m-\hat m|^{2}$, the desired lower bound for $1-(\bar\beta+\hat\beta)$ follows.

Next, we show that $f(\bar B,\bar C)$ is bounded below by a structural constant.  In fact, from \cite[first identity in (4.7)]{gutierrez-huang:nearfieldrefractor}, $f(\bar B,\bar C)=h(x,\bar Y,X_0)\geq \dfrac{\bar b-\kappa\, |\bar Y|}{1+\kappa}$ for all $x\in S^{n-1}$ where 
$\bar b=|X_0|+\kappa\,|\bar Y-X_0|$. So $\dfrac{\bar b-\kappa\, |\bar Y|}{1+\kappa}\geq \dfrac{1-\kappa}{1+\kappa}|X_0|\geq c_1\dfrac{1-\kappa}{1+\kappa}$, since $X_0\in \Gamma_{c_1c_2}$.
Thus,
\begin{align*}
sf(\bar B,\bar C)(1-(\bar\beta+\hat\beta))|x-x_0|^{2}&\geq C\,s\lambda(1-\lambda)|\bar m-\hat m|^{2}|x-x_0|^{2}\\
&\geq
C\,s\lambda(1-\lambda)|\bar Y-\hat Y|^{2}|x-x_0|^{2}\quad \text{from \ref{item:hypotheses A}(c)}
\end{align*}
with $C>0$ a structural constant.

It remains to estimate $B+\sqrt{B^{2}-C}-f(\bar B,\bar C)$ from above. 
We have from \eqref{eq:expression for b -kappa2xcdotY} that $B+\sqrt{B^{2}-C}-f(\bar B,\bar C)\leq B+\sqrt{B^{2}-C}\leq 2B\leq C\,\(|X_0|+\kappa \,s\)$. Since $s=s_{X_0}(m)=|Y-X_0|$ we obtain from \ref{item:old hypotheses C} that $|X_0|+\kappa \,s\leq C\,s$ with a structural constant $C>0$.
Therefore
\[
\dfrac{E}{B+\sqrt{B^{2}-C}-f(\bar B,\bar C)}\geq C\,\lambda\,(1-\lambda)|\bar Y-\hat Y|^{2}|x-x_0|^{2}
\]
for $0<\lambda <1$ with $C>0$ a structural constant (since $\mu<\kappa$). The proof of the lemma is then complete.
\end{proof}

\section{Estimates for derivatives of ovals}\label{sec:estimates for derivatives of ovals}
\setcounter{equation}{0}
We analyze now the derivatives of the function $h(x,Y,X_0)$ for $Y\in\Sigma$ and $X_0\in\Gamma_{c_1c_2}$.
To differentiate the function $h$ with respect to the variables $x$ and $Y$ we will extend $h(x,Y,X_0)$ 
for $x$ in a neighborhood of the unit ball and $Y$ in a neighborhood of $\Sigma$. 
In order to do this, we first need to bound from below the quantity inside the square root in \eqref{eq:polar radius of the oval}.

\begin{lemma}\label{aux4}

Let $X_0\in \Gamma_{c_1c_2}$. There exist $\epsilon>0$ sufficiently small depending only on $\kappa$ and constants 
$C_0,C_1$ depending only on $\kappa$ and $c_2$ such that if $b=|X_0|+\kappa|Y-X_0|<|Y|$ and $|Y|\geq C_1$ then
\begin{equation}\label{eq:lower estimated of Delta in square root}
\(b-\kappa^{2}\, x\cdot Y\)^{2}-\(1-\kappa^{2}\)\(b^{2}-\kappa^{2}|Y|^{2}\)\geq C_0\qquad \text{for all $|x|\leq 1+\epsilon$}.
\end{equation}
Then by continuity there is a small neighborhood $V$ of $Y$ such that \eqref{eq:lower estimated of Delta in square root}
holds for all $Y\in V$ with a smaller positive constant $C$. 
This implies that under this configuration, the formula defining $h$ in \eqref{eq:definition of function h} can be extended for $|x|\leq 1+\epsilon$ and $Y\in V$.

\end{lemma}
\begin{proof}
By calculation 
\begin{equation}\label{eq:definition of Deltat}
\Delta(t):=\(b-\kappa^{2}\, t\)^{2}-\(1-\kappa^{2}\)\(b^{2}-\kappa^{2}|Y|^{2}\)
=
\kappa^{2}\((b- t)^{2}
+(1-\kappa^{2})\(|Y|^{2}-t^{2}\)\).
\end{equation}
From the estimate for the ovals \cite[first identity in (4.7)]{gutierrez-huang:nearfieldrefractor}, $|X_0|=h(x_0,Y,X_0)\geq \dfrac{b-\kappa\, |Y|}{1+\kappa}$, so $b\leq \kappa\,|Y|+(1+\kappa)|X_0|$.
Therefore $|Y|-b\geq (1-\kappa)\,|Y|-(1+\kappa)\,|X_0|$. Clearly, the last quantity is non negative if
$|Y|\geq \dfrac{1+\kappa}{1-\kappa}|X_0|$.

We have $\displaystyle \min_{|x|\leq 1+\epsilon}\Delta(x\cdot Y)=\min_{-(1+\epsilon)|Y|\leq t \leq (1+\epsilon)|Y|}\Delta(t)$. The function $\Delta(t)$ is decreasing in the interval $\(-\infty,b/\kappa^2\)$. Let $\epsilon>0$ be such that $\kappa\,(1+\epsilon)<1$. Since $b>\kappa\,|Y|$ we then have $\left[-(1+\epsilon)|Y|, (1+\epsilon)|Y|\right]
\subset \(-\infty,b/\kappa^2\)$. Therefore 
\[
\min_{-(1+\epsilon)|Y|\leq t \leq (1+\epsilon)|Y|}\Delta(t)=\Delta\((1+\epsilon)|Y|\)
\]
Let us estimate $\Delta\((1+\epsilon)|Y|\)$ from below: 
\begin{align*}
&\Delta((1+\epsilon) |Y|)\\
&=\kappa^2\,\(\(b- (1+\epsilon)|Y|\)^2+(1-\kappa^2)|Y|^2\(1-(1+\epsilon)^2\)\)\\
&=\kappa^2\,
\(b^2-2\,b\,|Y|-2\,b\,\e\,|Y|+\(1+\kappa^2\,\e\,(2+\e)\)\,|Y|^2 \)\\
&=\kappa^2\,
\(\(|Y|-b\)^2+\kappa^2\,\e\,(2+\e)\,|Y|^2-2\,b\,\e\,|Y| \)\\
&\geq \kappa^2\,
\(\((1-\kappa)\,|Y|-(1+\kappa)\,|X_0|\)^2+\kappa^2\,\e\,(2+\e)\,|Y|^2-2\,\(\kappa\,|Y|+(1+\kappa)|X_0|\)\,\e\,|Y| \)\\
&=
\kappa^2\((1-\kappa)^2\,|Y|^2-2(1-\kappa^2)|Y|\,|X_0|+(1+\kappa)^2\,|X_0|^2+\kappa^2\,\e\,(2+\e)\,|Y|^2-
2\,
\kappa\,\epsilon\,|Y|^2
-
2\,(1+\kappa) \e\,|X_0|\,|Y|\)\\
&=
\kappa^2\(\((1-\kappa)^2+\kappa^2\,\e\,(2+\e)-2\,
\kappa\,\epsilon\)\,|Y|^2 - 2\(1-\kappa^2+(1+\kappa) \e\)\,|Y|\,|X_0| +(1+\kappa)^2\, |X_0|^2 \)\\
&=
\kappa^2\(\alpha_1\,|Y|^2 - 2\,\alpha_2\,|Y|\,|X_0| +(1+\kappa)^2\, |X_0|^2 \).
\end{align*}
From the choice of $\epsilon$, $\alpha_2\leq (1-\kappa^2)\dfrac{1+\kappa}{\kappa}:=\beta_2$ and taking $\epsilon$ small we have $\alpha_1\geq (1-\kappa)^2/2:=\beta_1$.
Hence
\begin{align*}
\Delta((1+\epsilon) |Y|)&
\geq
\kappa^2\(\beta_1\,|Y|^2 - 2\,\beta_2\,|Y|\,|X_0| +(1+\kappa)^2\, |X_0|^2 \)\\
&\geq
\kappa^2 \,\(\beta_1\,|Y|^2 - \beta_2\,\(\delta\,|Y|^2+\dfrac{|X_0|^2}{\delta}\) +(1+\kappa)^2\, |X_0|^2\)\\
&=
\kappa^2 \,\(\(\beta_1-\delta\,\beta_2\)\,|Y|^2 - \(\dfrac{\beta_2}{\delta}-(1+\kappa)^2\)\, |X_0|^2\),\qquad \delta>0. 
\end{align*}
We now choose $\delta>0$ sufficiently small depending only on $\kappa$ such that
\[
\(\beta_1-\delta\,\beta_2\)\,|Y|^2 - \(\dfrac{\beta_2}{\delta}-(1+\kappa)^2\)\, |X_0|^2
\geq 
C_1(\kappa)\,|Y|^2 - C_2(\kappa)\, |X_0|^2,
\]
for some $C_i$ positive constants. 
Thus
\[
\Delta((1+\epsilon) |Y|)\geq \(\sqrt{C_1(\kappa)}\,|Y| + \sqrt{C_2(\kappa)}\, |X_0|\)\(\sqrt{C_1(\kappa)}\,|Y| - \sqrt{C_2(\kappa)}\, |X_0|\),
\]
and the desired inequality follows.

\end{proof}

With Lemma \ref{aux4} in hand we proceed to prove estimates for $h$ and its derivatives.

\begin{lemma}\label{aux5}
There exists a structural constant $C>0$   
such that
if $Y\in\Sigma$, $t>0$ and $(1+t)X_0\in \Gamma_{c_1c_2}$, then $0\leq h(x,Y,(1+t)X_0)-h(x,Y,X_0)\leq C\,t\,|X_0|.$
\end{lemma}
\begin{proof}
If $b(t)=(1+t)|X_0|+\kappa|Y-(1+t)X_0|$, then $0\leq b(t)-b(0)\leq (1+\kappa)t|X_0|$.
\footnote{We can see $b(0)\leq b(t)$ for $t>0$ because this is equivalent to $(1+t)|X_0|+\kappa|Y-(1+t)X_0|\geq |X_0|+\kappa|Y-X_0|$ which is equivalent to show $t|X_0|+\kappa|Y-(1+t)X_0|\geq \kappa|Y-X_0|$. But by triangle inequality $\kappa|Y-(1+t)X_0|\geq \kappa |Y-X_0|-\kappa t|X_0|$ which implies
$t|X_0|+\kappa|Y-(1+t)X_0|\geq t|X_0|+\kappa |Y-X_0|-\kappa t|X_0|= (1-\kappa)t|X_0|+\kappa |Y-X_0|>\kappa |Y-X_0|$ since $\kappa<1$.}
Let $Q(t)=(b(t)-\kappa^{2}\, x\cdot Y)^{2}-(1-\kappa^{2})(b(t)^{2}-\kappa^2\,|Y|^{2})$.
We have
\begin{align*}
Q(0)-Q(t)
&=
(b(0)-\kappa^{2}\, x\cdot Y)^{2}-(b(t)-\kappa^{2}\, x\cdot Y)^{2}-(1-\kappa^{2})(b(0)^{2}-b(t)^{2})\\
&=
\kappa^2\,(b(0)^{2}-b(t)^{2})-2\,\kappa^2\,x\cdot Y\,\(b(0)-b(t)\)\\
&=
\kappa^2\(b(0)-b(t)\)\,\(b(0)+b(t)-2\,x\cdot Y\).
\end{align*}
From the definition of $h$
\begin{align*}
&h(x,Y,(1+t)X_0)-h(x,Y,X_0)\\&=\dfrac{b(t)-\kappa^2\,x\cdot Y-\sqrt{Q(t)}}{1-\kappa^2}-\dfrac{b(0)-\kappa^2\,x\cdot Y-\sqrt{Q(0)}}{1-\kappa^2}\\
&=\dfrac{b(t)-b(0)+\sqrt{Q(0)}-\sqrt{Q(t)}}{1-\kappa^2}\\
&=
\dfrac{1}{1-\kappa^2}\(\dfrac{\(b(t)-b(0)\)\(\sqrt{Q(0)}+\sqrt{Q(t)}\)+Q(0)-Q(t)}{\sqrt{Q(0)}+\sqrt{Q(t)}} \)\\
&=
\dfrac{1}{1-\kappa^2}\(\dfrac{\(b(t)-b(0)\)\(\sqrt{Q(0)}+\sqrt{Q(t)}\)+\kappa^2\(b(0)-b(t)\)\,\(b(0)+b(t)-2\,x\cdot Y\)}{\sqrt{Q(0)}+\sqrt{Q(t)}} \)\\
&=\dfrac{(b(t)-b(0))\,\(\sqrt{Q(t)}+\sqrt{Q(0)}+\kappa^{2}\(b(0)+b(t)-2\,x\cdot Y\)\)}{\(1-\kappa^2\)\(\sqrt{Q(t)}+\sqrt{Q(0)}\)}.
\end{align*}
Since $\rho(x,Y,b)$ is increasing in $b$ and $b(0)<b(t)$, it follows from \eqref{eq:definition of function h} that $h(x,Y,X_0)\leq h(x,Y,(1+t)X_0)$.
From Lemma \ref{aux4} the denominator in the last string of expressions is bounded away from zero and we obtain
\[
0\leq h(x,Y,(1+t)X_0)-h(x,Y,X_0)\leq C\,(b(t)-b(0))\leq C\,t\,|X_0|.
\]
\end{proof}
\begin{lemma}\label{aux6}
Suppose \ref{item:hypotheses A} and \ref{item:old hypotheses C} hold.
There exist a structural constant $C>0$ such that if $\bar Y,Y\in\Sigma$ and $X_0\in \Gamma_{c_1c_2}$, then $|\nabla_{x}h(x_0,Y,X_0)-\nabla_{x}h(x_0,\bar Y,X_0)|\leq C\,|Y-\bar Y|$. 
\end{lemma}
\begin{proof}
Let $Y=X_0+sm$ and $\bar Y=X_0+\bar s\bar m$ and $b=|X_0|+\kappa s$.
From Lemma \ref{aux4} we can take derivatives of $h$ with respect to $x$ for $x$ in a neighborhood of the unit ball, 
and by calculation
\begin{equation}\label{eq:partial h with respect to xi}
\frac{\partial h}{\partial x_{i}}(x,Y,X_0)
=\dfrac{\kappa^{2}h(x,Y,X_0)Y_{i}}{\sqrt{(b-\kappa^{2}\, x\cdot Y)^{2}-(1-\kappa^{2})(b^{2}-\kappa^2\,|Y|^{2})}}
=
\dfrac{\kappa^{2}h(x,Y,X_0)Y_{i}}{\sqrt{\Delta \(x\cdot Y\)}}.
\end {equation}
So at $x=x_0$
\begin{equation}\label{eq:partial h with respect to xi at x0}
\frac{\partial h}{\partial x_{i}}(x_0,Y,X_0)=\dfrac{\kappa^{2}|X_0|Y_{i}}{\sqrt{(b-\kappa^{2}\, x_0\cdot Y)^{2}-(1-\kappa^{2})(b^{2}-\kappa^2\,|Y|^{2})}}=
\dfrac{\kappa^{2}|X_0|Y_{i}}{\kappa s(1-\kappa\, x_0\cdot m)},
\end{equation}
since 
$
\sqrt{(b-\kappa^{2}\, x_0\cdot Y)^{2}-(1-\kappa^{2})(b^{2}-\kappa^2\,|Y|^{2})}=\kappa s(1-\kappa\, x_0\cdot m)
$
from \eqref{eq:expression for b -kappa2xcdotY} and \eqref{eq:expression for b2 -kappa2Y2}.

Therefore
\begin{align*}
&\nabla_{x}h(x_0,Y,X_0)-\nabla_{x}h(x_0,\bar Y,X_0)\\
&=\kappa|X_0|\,
\(\dfrac{Y}{|Y-X_0|(1-\kappa\, x_0\cdot m)}-\dfrac{\bar Y}{|\bar Y-X_0|(1-\kappa\, x_0\cdot \bar m)}\)\\
&=
\kappa|X_0|\,
\(\dfrac{Y-\bar Y}{|Y-X_0|(1-\kappa\, x_0\cdot m)}
+
\bar Y
\(\dfrac{1}{|Y-X_0|(1-\kappa\, x_0\cdot m)}
-\dfrac{1}{|\bar Y-X_0|(1-\kappa\, x_0\cdot \bar m)}\)\)\\
&=\kappa|X_0|\,\(A+B\).
\end{align*}
From \eqref{eqlower bound for s_X}, $|A|\leq C\,|Y-\bar Y|$.
In addition 
\begin{align*}
|B|&\leq C\,\left| |\bar Y-X_0|(1-\kappa\, x_0\cdot \bar m)-|Y-X_0|(1-\kappa\, x_0\cdot m)\right|\\
&=C\,\left| |\bar Y-X_0|-|Y-X_0| +\kappa\,|Y-X_0|\, x_0\cdot m-\kappa\, |\bar Y-X_0|\,x_0\cdot \bar m\right|\\
&= C\,\left| |\bar Y-X_0|-|Y-X_0|+\kappa \,|Y-X_0| \,x_0\cdot (m-\bar m)+\kappa\,x_0\cdot \bar m \(|Y-X_0|-|\bar Y-X_0|\)\right|\\
&\leq C\,|Y-\bar Y|
\end{align*}
since  $|m-\bar m|\leq C|Y-\bar Y|$ from \eqref{eq:difference of ms bounded by difference of Ys}.
\end{proof}
\begin{lemma}\label{aux7}
There exists a structural constant $M$ such that if $X_0\in\Gamma_{c_1c_2}$, $Y\in\Sigma$ and $x_0=X_0/|X_0|$, then 
\[
|h(x,Y,X_0)-h(x_0,Y,X_0)-\langle\nabla_{x} h(x_0,Y,X_0),x-x_0\rangle|\leq M|x-x_0|^{2}
\]
for all $x\in S^{n-1}$.
\end{lemma}
\begin{proof}
We first calculate $\dfrac{\partial^2}{\partial x_j\partial x_i}$. From \eqref{eq:definition of Deltat} and \eqref{eq:partial h with respect to xi} 
\begin{align*}
&\frac{\partial^2 h}{\partial x_j \partial x_{i}}(x,Y,X_0)\\
&=
\dfrac{\partial }{\partial x_j}\(\kappa^{2}h(x,Y,X_0)Y_i\,\Delta(x\cdot Y)^{-1/2} \)\\
&=\kappa^4 \,h(x,Y,X_0)\,Y_i\,Y_j\,\Delta(x\cdot Y)^{-1}-\dfrac{\kappa^2}{2}\,h(x,Y,X_0)\,Y_i\,\Delta(x\cdot Y)^{-3/2}\dfrac{\partial \Delta}{\partial x_j}\\
&=\kappa^4 \,h(x,Y,X_0)\,Y_i\,Y_j\,\Delta(x\cdot Y)^{-1}-\dfrac{\kappa^2}{2}\,h(x,Y,X_0)\,Y_i\,\Delta(x\cdot Y)^{-3/2}(-2\,\kappa^2 (b-\kappa^2\,x\cdot Y)\,Y_j)\\
&=\kappa^4 \,h(x,Y,X_0)\,Y_i\,Y_j\,\Delta(x\cdot Y)^{-1}
\(1+\dfrac{b-\kappa^2\,x\cdot Y}{\sqrt{\Delta(x\cdot Y)}}\).
\end{align*}
From Lemma \ref{aux4} we obtain that $\left| \dfrac{\partial^2 h}{\partial x_j \partial x_{i}}(x,Y,X_0)\right|\leq C_1$ for all $x$ in a neighborhood of the unit ball $|x|\leq 1$.
From Taylor's formula
\[
h(x,Y,X_0)=h(x_0,Y,X_0)+\nabla_x h(x_0,Y,X_0)\cdot (x-x_0)
+
\dfrac12 \,\left\langle D^2_xh(\xi,Y,X_0)(x-x_0),x-x_0\right\rangle 
\]
with $\xi$ between $x_0$ and $x$.
The lemma then follows.
\end{proof}
\begin{lemma}\label{aux8}
Suppose \ref{item:hypotheses A} and \ref{item:old hypotheses C} hold.
There exists a structural constant $C>0$ such that if $X_0\in \Gamma_{c_1c_2}$ and $\bar Y,Y\in\Sigma$,  then
\[
|h(x,Y,X_0)-h(x,\bar Y,X_0)|\leq C\,|Y-\bar Y|\,|x-x_0|,
\]
for $x\in S^{n-1}$.
\end{lemma}
\begin{proof}

Using Lemma \ref{aux4} we shall first estimate the derivatives of $h$ with respect to $Y_k$; $Y=(Y_1,\cdots ,Y_n)$.
Recall $h(x,Y,X_0)=\dfrac{1}{1-\kappa^2}\(b-\kappa^2\,x\cdot Y-\sqrt{\Delta(x\cdot Y, b, |Y|)}\)$ where 
$\Delta(t,b,|Y|)=\Delta(t)$ is given by \eqref{eq:definition of Deltat} and $b=|X_0|+\kappa\,|X_0-Y|$. By Lemma \ref{aux4}, $h$ can be differentiated with respect to $Y_k$ since is defined in an open neighborhood of the target $\Sigma$.
Then 
\begin{align*}
\dfrac{\partial h}{\partial Y_k}&=\dfrac{1}{1-\kappa^2}\(\dfrac{\partial b}{\partial Y_k}-\kappa^2\,x_k-\dfrac12 \,\Delta^{-1/2}\,\dfrac{\partial \Delta}{\partial Y_k} \).
\end{align*}
Now 
\[
\dfrac{\partial \Delta}{\partial Y_k}
=
2\(b-\kappa^2\,x\cdot Y\)\(\dfrac{\partial b}{\partial Y_k}-\kappa^2\,x_k\)
-\(1-\kappa^2\)\(2\,b\,\dfrac{\partial b}{\partial Y_k}-2\,\kappa^2\,Y_k\),
\]
and
$
\dfrac{\partial b}{\partial Y_k}
=
-\kappa\,\dfrac{X_0^k-Y_k}{|X_0-Y|}$.
We next differentiate \eqref{eq:partial h with respect to xi} with respect to $Y_k$. Recall that from Lemma \ref{aux4}, the right hand side of \eqref{eq:partial h with respect to xi} is well defined for $x$ in a neighborhood of the unit ball and for $Y$ in a neighborhood of the target $\Sigma$.
We then have
\begin{align*}
\dfrac{\partial^{2}h}{\partial Y_k\partial x_i}(x,Y,X_0)
=
\kappa^2\,\dfrac{\partial h}{\partial Y_k}\,Y_i\,\Delta^{-1/2} 
+
\kappa^2\,h\,\delta_{ik}\,\Delta^{-1/2}
-
\dfrac12 \,\kappa^2\,h\,Y_i\,\Delta^{-3/2}\,\dfrac{\partial \Delta}{\partial Y_k}.
\end{align*}
From Lemma \ref{aux4}, $\Delta\geq C$ so $\dfrac{\partial h}{\partial Y_k}$ is bounded, and therefore $\dfrac{\partial^{2}h}{\partial Y_k\partial x_i}(x,Y,X_0)$ is also bounded.

Therefore we can write for some $\tilde Y\in\overline{\bar YY}$, the straight segment, and for some $\tilde x\in \overline{x_0x}$
\begin{align*}
h(x,Y,X_0)-h(x,\bar Y,X_0)&=\sum_{k=1}^{n}\frac{\partial h}{\partial Y_k}(x,\tilde Y,X_0)(Y_k-\bar Y_k)\\
&=
\sum_{k=1}^{n}\(\frac{\partial h}{\partial Y_k}(x,\tilde Y,X_0)-\frac{\partial h}{\partial Y_k}(x_0,\tilde Y,X_0)\)(Y_k-\bar Y_k)\\
&=
\sum_{k,l=1}^{n}\frac{\partial^{2}h}{\partial Y_k\partial x_l}(\tilde x,\tilde Y,X_0)(x_l-x^{0}_l)(Y_k-\bar Y_k)
\end{align*}
where we have used that $h(x_0,Y,X_0)=|X_0|$, for all $Y$ so $\dfrac{\partial h}{\partial Y_k}(x_0,\tilde Y,X_0)=0$.
It remains to show that $\Delta(x\cdot \tilde Y,\tilde b,|\tilde Y|)\geq C$ and $\Delta(\tilde x\cdot \tilde Y,\tilde b, |\tilde Y|)\geq C$ so the application of the mean value theorem above is justified and we can apply the bounds for the derivatives.
We have $\tilde Y=(1-\lambda)\bar Y+\lambda Y$ for some $\lambda\in[0,1]$.
From \ref{item:hypotheses A}, we can write $Y=X_0+s\,m$ and $\bar Y=X_0+\bar s\,\bar m$ with $x_0\cdot \bar m\geq \kappa$, $x_0\cdot m\geq \kappa$, $x_0=X_0/|X_0|$. So $\tilde Y=X_0+(1-\lambda)\bar s\,\bar m+\lambda s\,m:=X_0+w$, and
\[
|\tilde Y|^{2}-b^{2}=|X_0+w|^{2}-(|X_0|+\kappa\,|w|)^{2}=|w|^{2}(1-\kappa^{2})+2\,X_0\cdot w-2\,\kappa|X_0|\,|w|
\]
and 
\[
X_0\cdot w=(1-\lambda)\bar s\, X_0\cdot \bar m+\lambda s\, X_0\cdot m\geq (1-\lambda)\bar s\kappa|X_0|+\lambda s\kappa|X_0|\geq \kappa|X_0|\,|w|.
\]
Thus
\begin{align*}
|\tilde Y|^{2}-b^{2}&\geq (1-\kappa^{2})|w|^{2}\\
&=\(1-\kappa^2\)\( (1-\lambda)^2\bar s^2
+2(1-\lambda)\lambda \,\bar s\, s\, m\cdot \bar m+\lambda^2 s^2\):=\(1-\kappa^2\)\,\varphi(\lambda).
\end{align*}
Since $\bar m\cdot x_0\geq \kappa$ and $m\cdot x_0\geq \kappa$ with $\kappa<1$, it follows that $\bar m\cdot m\geq -\delta$ for some $0<\delta=\delta(\kappa)<1$.
Then 
\[
\varphi(\lambda)\geq (1-\lambda)^2\bar s^2
-2(1-\lambda)\lambda \,\bar s\, s\, \delta+\lambda^2 s^2,
\]
where the last expression attains its minimum when $\lambda=\dfrac{\bar s^2+\delta \,\bar s \, s}{\bar s^2+2\,\delta \,\bar s\,s+s^2}$. 
Since $\bar s,s$ are bounded, at this minimum the expression is larger than or equal to $C\,(1-\delta^2)\,\min\left\{\bar s^2,s^2\right\}$, with $C>0$ structural.
From \eqref{eqlower bound for s_X} we then obtain
\[
|\tilde Y|^{2}-b^{2}\geq C>0.
\]
Using the argument the proof of Lemma \ref{aux4} with $\epsilon=0$, it follows that $\Delta\(x\cdot \tilde Y,b,|\tilde Y|\)$ and $\Delta\(\tilde x\cdot \tilde Y,b,|\tilde Y|\)$ are both greater than or equal to $\kappa^2\(|\tilde Y|-b \)^2$ obtaining the desired estimate.

\end{proof}

\section{$C^{1,\alpha}$ estimates}\label{sec:C 1 alpha estimates for near field refractors}
\setcounter{equation}{0}

We now turn to the definition of refractor and prove our main theorem.
\begin{definition}\label{def:definition of refractor}
We say $u:\Omega\rightarrow[c_1,c_2]$ is a refractor from $\Omega$ to $\Sigma$ if for each $x_0\in\Omega$, there exists $Y\in\Sigma$ such that
\[
u(x)\geq h(x,Y,X_0)
\]
for all $x\in\Omega$ with $X_0=u(x_0)x_0$.
If this holds, then we say $Y\in \partial u(x_0)$.
Notice that $X=u(x)x\in \Gamma_{c_1c_2}$ for all $x\in\Omega$.
\end{definition}

We will show $u\in C^{1,\alpha}(\Omega)$, which will follow from the following two lemmas.

\begin{lemma}\label{main1}
Assume \ref{item:hypotheses A}, \ref{item:old hypotheses C}, and \ref{item:old hypotheses D}, and let $u$ be a refractor from $\Omega$ to $\Sigma$.
There exist structural constants $K_1,K_2$ 
such that 
if $B_{2\delta}\cap S^{n-1}\subseteq\Omega$, $\bar x,\hat x\in B_{\delta}\cap S^{n-1}$, $\bar Y\in\partial u(\bar x)$ and $\hat Y\in\partial u(\hat x)$, with $|\bar Y-\hat Y|\geq |\bar x-\hat x|$, then, there exists $x_0\in B_{\delta}\cap S^{n-1}$ such that, letting $X_0=u(x_0)x_0$, if $Y(\lambda)\in [\bar Y,\hat Y]_{X_0}$ we have 
\[
u(x)\geq h(x,Y(\lambda),X_0)+K_1\lambda(1-\lambda)|\bar Y-\hat Y|^{2}|x-x_0|^{2}-K_2|\bar x-\hat x||\bar Y-\hat Y|^{2}
\]
for all $x\in\Omega$, $0<\lambda<1$.
\end{lemma}

\begin{proof}
Let $\bar X=u(\bar x)\bar x$ and $\hat X=u(\hat x)\hat x$. 
We have $u(x)\geq h(x,\bar Y,\bar X)$ and $u(x)\geq h(x,\hat Y,\hat X)$, for all $x\in\Omega$.
Let $\varphi(x)=h(x,\bar Y,\bar X)-h(x,\hat Y,\hat X)$. Since $\varphi(\bar x)\geq 0$ and $\varphi(\hat x)\leq 0$, by continuity there exists $x_0\in [\bar x,\hat x]$, the geodesic segment in the unit sphere, such that 
$h(x_0,\bar Y,\bar X)=h(x_0,\hat Y,\hat X):=\rho_0$.
Set $\tilde X_0=\rho_0x_0$ and $X_0=u(x_0)x_0$ and notice $\rho_0\leq u(x_0)$ and by definition of refractor $C_1\leq u(x_0)\leq C_2$, i.e., $X_0\in \Gamma_{C_1C_2}$.
Also, the oval with focus $\bar Y$ that passes through $\bar X$ then also passes through $\tilde X_0$, i.e., 
$h(x,\bar Y,\bar X)=h(x,\bar Y,\tilde X_0)$ for all $x\in S^{n-1}$; and similarly $h(x,\hat Y,\hat X)=h(x,\hat Y,\tilde X_0)$.
Hence $h\(x_0,\bar Y,\bar X\)=h\(x_0,\bar Y,\tilde X_0\)=|\tilde X_0|$. From 
\cite[first identity in (4.7)]{gutierrez-huang:nearfieldrefractor}, $h\(x_0,\bar Y,\bar X\)\geq \dfrac{b-\kappa\,|\bar Y|}{1+\kappa}$ where $b=|\bar X|+\kappa\,|\bar X-\bar Y|$. Therefore, $h\(x_0,\bar Y,\bar X\)\geq \dfrac{1-\kappa}{1+\kappa}|\bar X|=\dfrac{1-\kappa}{1+\kappa}\,u(\bar x)\geq \dfrac{1-\kappa}{1+\kappa}\,C_1$.
Thus, $|\tilde X_0|\geq \dfrac{1-\kappa}{1+\kappa}\,C_1$.

We claim
\[
u(x_0)-\rho_0\leq C\,|\bar x-\hat x|\,|\bar Y-\hat Y|
\]
for some structural constant $C$.
Suppose for a moment the claim holds true.
We can write $X_0=(1+t)\tilde X_0\in \Gamma_{C_1C_2}$ with $t=\dfrac{u(x_0)-\rho_0}{\rho_0}$. So 
applying Lemma \ref{aux5} yields
\[
 h(x,\bar Y,\bar X)=h(x,\bar Y,\tilde X_0)\geq h(x,\bar Y,X_0)-C(u(x_0)-\rho_0)\geq h(x,\bar Y,X_0)-C|\bar x-\hat x||\bar Y-\hat Y|
\]
and 
\[
h(x,\hat Y,\bar X)=h(x,\hat Y,\tilde X_0)\geq h(x,\hat Y,X_0)-C(u(x_0)-\rho_0)\geq h(x,\hat Y,X_0)-C|\bar x-\hat x||\bar Y-\hat Y|
\]
for all $x\in\Omega$.
Thus
\begin{align*}
u(x)&\geq\max\{h(x,\bar Y,\tilde X_0),h(x,\hat Y,\tilde X_0)\}\\
&\geq\max\{h(x,\bar Y, X_0),h(x,\hat Y,X_0)\}-C|\bar x-\hat x||\bar Y-\hat Y|\\
&\geq
h(x,Y(\lambda),X_0)+K_1\,\lambda(1-\lambda)|\bar Y-\hat Y|^{2}|x-x_0|^{2}-K_2\,|\bar x-\hat x||\bar Y-\hat Y|,
\end{align*}
where in the last inequality we have used Lemma \ref{crucial} and renamed the resulting constants.

It then remains to prove the claim.
Since $x_0\in [\bar x,\hat x]$, we can write $x_0=\dfrac{(1-t)\bar x+t\hat x}{|(1-t)\bar x+t\hat x|}:=\dfrac{x_t}{|x_t|}$, for some $t\in [0,1]$.
If $Y_0\in \partial u(x_0)$, then
\begin{align*}
u(x)&\geq h(x,Y_0,X_0)\geq h(x_0,Y_0,X_0)+\langle\nabla_{x}h(x_0,Y_0,X_0),x-x_0\rangle-M|x-x_0|^{2}\\
&=
u(x_0)+\langle\nabla_{x}h(x_0,Y_0,X_0),x-x_0\rangle-M|x-x_0|^{2},
\end{align*}
from Lemma \ref{aux7}.
Therefore
\begin{align}\label{eq:lower estimate convex combinantion}
&(1-t)u(\bar x)+tu(\hat x)\\
&\geq u(x_0)+\langle\nabla_{x}h(x_0,Y_0,X_0),x_t-x_0\rangle
-
M\((1-t)|\bar x-x_0|^{2}+t|\hat x-x_0|^{2}\).\notag
\end{align}
By calculation
\begin{equation}\label{eq:bound for |xt|}
(1-t)|\bar x-x_0|^{2}+t|\hat x-x_0|^{2}=2(1-|x_t|)=2\,\dfrac{1-|x_t|^2}{1+|x_t|}\leq 2\,|\bar x-\hat x|^{2}.
\end{equation}
From \eqref{eq:partial h with respect to xi} $|\nabla_{x}h(x_0,Y_0,X_0)|\leq C$ and since $|x_t-x_0|\leq 2\,|\bar x-\hat x|^{2}$ it then follows from \eqref{eq:lower estimate convex combinantion} that
\[
u(x_0)\leq (1-t)u(\bar x)+tu(\hat x)+C|\bar x-\hat x|^{2}.
\]
Next, since as proved above, $|\tilde X_0|\geq \dfrac{1-\kappa}{1+\kappa}\,C_1$, we can apply Lemma \ref{aux7}
with $X_0$ replaced by $\tilde X_0$ to obtain
\begin{align*}
u(\bar x)&=h(\bar x,\bar Y,\tilde X_0)\leq h(x_0,\bar Y,\tilde X_0)+\langle\nabla_{x}h(x_0,\bar Y,\tilde X_0),\bar x-x_0\rangle+M|\bar x-x_0|^{2}\\
&=
\rho_0+\langle\nabla_{x}h(x_0,\bar Y,\tilde X_0),\bar x-x_0\rangle+M|\bar x-x_0|^{2},
\end{align*}
and similarly
\[
u(\hat x)\leq\rho_0+\langle\nabla_{x}h(x_0,\hat Y,\tilde X_0),\hat x-x_0\rangle+M|\hat x-x_0|^{2}.
\]
Therefore
\begin{align*}
(1-t)u(\bar x)+tu(\hat x)&\leq \rho_0+(1-t)\langle\nabla_{x}h(x_0,\bar Y,\tilde X_0),\bar x-x_0\rangle+t\langle\nabla_{x}h(x_0,\hat Y,\tilde X_0),\hat x-x_0\rangle\\
&\qquad +M\, \((1-t)|\bar x-x_0|^{2}+t|\hat x-x_0|^{2}\).
\end{align*}
The last term is bounded above by $2\,M\,|\bar x-\hat x|^{2}$.
To estimate the middle term we write
\[
\bar x-x_0=\dfrac{\bar x(|x_t|-1)-t(\hat x-\bar x)}{|x_t|},\qquad
\hat x-x_0=\dfrac{\hat x(|x_t|-1)+(1-t)(\hat x-\bar x)}{|x_t|},
\]
so
\begin{align*}
&(1-t)\langle\nabla_{x}h(x_0,\bar Y,\tilde X_0),\bar x-x_0\rangle+t\langle\nabla_{x}h(x_0,\hat Y,\tilde X_0),\hat x-x_0\rangle\\
&=
\dfrac{(1-t)\,t}{|x_t|}\,\left\langle\nabla_{x}h(x_0,\hat Y,\tilde X_0)-\nabla_{x}h(x_0,\bar Y,\tilde X_0),\hat x-\bar x\right\rangle\\
&\qquad +
\dfrac{|x_t|-1}{|x_t|}\(\left\langle\,(1-t)\,\nabla_{x}h(x_0,\bar Y,\tilde X_0),\bar x\right\rangle+\left\langle\,t\,\nabla_{x}h(x_0,\hat Y,\tilde X_0),\hat x\right\rangle\).
\end{align*} 
Since $|\tilde X_0|\geq \dfrac{1-\kappa}{1+\kappa}\,C_1$, from \eqref{eq:bound for |xt|} and \eqref{eq:partial h with respect to xi} the absolute value of the last term is $\leq C|\bar x-\hat x|^{2}$;
and we can apply Lemma \ref{aux6} to obtain that the absolute value of the first term is bounded by $ C|\bar Y-\hat Y||\bar x-\hat x|$. Since $|\bar x-\hat x|\leq |\bar Y-\hat Y|$, the claim is proved, and the lemma follows.

\end{proof}

Now using Lemmas \ref{main1} and \ref{aux8}, we obtain the following.
\begin{lemma}\label{main2}
Under the hypotheses of Lemma \ref{main1}, there exist structural constants $K_1,K_2,K_3$ and $x_0\in B_{\sigma}\cap S^{n-1}$ such that for all $Y(\lambda)\in[\bar Y,\hat Y]_{X_0}$, $Y\in\Sigma$ and $x\in\Omega$,
\[
u(x)\geq h(x,Y,X_0)+K_1\lambda(1-\lambda)|\bar Y-\hat Y|^{2}\,|x-x_0|^{2}-K_2|Y-Y(\lambda)|\,|x-x_0|-K_3|\bar Y-\hat Y|\,|\bar x-\hat x|
\]
where $X_0=u(x_0)x_0$, $0<\lambda<1$.
\end{lemma}
Our main theorem is then the following.
\begin{theorem}\label{thm:main}
Suppose that \ref{item:hypotheses A}, \ref{item:old hypotheses C}, \ref{item:old hypotheses D}, and \ref{item:old hypotheses E} hold.
Let $u$ be a refractor from $\Omega$ to $\Sigma$ and 
assume that there is a constant $C$ such that for all balls $B_{\sigma}$ such that $B_{\sigma}\cap S^{n-1}\subseteq\Omega$, we have
\begin{equation}\label{measure}
H^{n-1}\(\partial u\(B_{\sigma}\cap S^{n-1}\)\)\leq C\,\sigma^{n-1}.
\end{equation}
where $H^{n-1}$ is the $(n-1)$-dimensional Hausdorff measure in $\R^{n}$.

Assume $B_{2\delta}\cap S^{n-1}\subseteq\Omega$. There exist constants $\tilde C_1,\tilde C_2$ depending on $\delta$ and structure, such that if $\bar x,\hat x\in B_{\delta}\cap S^{n-1}$, $\bar Y\in\partial u(\bar x),\hat Y\in\partial u(\hat x)$ with $|\bar Y-\hat Y|\geq \tilde C_1|\bar x-\hat x|$, then $|\bar Y-\hat Y|\leq \tilde C_2\,|\bar x-\hat x|^{\alpha}$ where $\alpha=\dfrac{1}{4n-5}$, $n>1$.
\end{theorem}
\begin{proof}
By Lemma \ref{main2}, there exists $x_0\in[\bar x,\hat x]\subseteq B_{\delta}$, such that for all $Y(\lambda)\in [\bar Y,\hat Y]_{X_0}$ with $\frac{1}{4}\leq\lambda\leq\frac{3}{4}$, for all $Y\in\Sigma$ and for all $x\in\Omega$, we have
\[
u(x)\geq h(x,Y,X_0)+K_1\,|\bar Y-\hat Y|^{2}\,|x-x_0|^{2}-K_2\,|Y-Y(\lambda)|\,|x-x_0|-K_3\,|\bar Y-\hat Y|\,|\bar x-\hat x|,
\]
where $X_0=u(x_0)x_0$ and $K_i,i=1,2,3$ are structural constants.
Let 
$$t_0=\dfrac{K_2|Y-Y(\lambda)|+\sqrt{K_2^{2}|Y-Y(\lambda)|^{2}+4K_1K_3|\bar Y-\hat Y|^{3}|\bar x-\hat x|}}{2K_1|\bar Y-\hat Y|^{2}}.$$
If $|x-x_0|\geq t_0$, then $K_1|\bar Y-\hat Y|^{2}|x-x_0|^{2}-K_2|Y-Y(\lambda)||x-x_0|-K_3|\bar Y-\hat Y||\bar x-\hat x|\geq 0$.
Let 
$
\mu=\sqrt{|\bar Y-\hat Y|^{3}|\bar x-\hat x|}
$
and suppose $|Y-Y(\lambda)|\leq\mu$, then
\[
t_0\leq\dfrac{K_2+\sqrt{K_2^{2}+4K_1K_3}}{2K_1}\sqrt{\dfrac{|\bar x-\hat x|}{|\hat Y-\bar Y|}}:=K\sqrt{\dfrac{|\bar x-\hat x|}{|\hat Y-\bar Y|}}:=\sigma.
\]
Let $C\geq 1$ be large enough constant depending on $\delta$ and the structural constants such that $\dfrac{K}{\sqrt{C}}\leq\dfrac{\delta}{2}$ and $\dfrac{(\diam(\Sigma))^{2}}{\sqrt C}\leq \mu_0$, with $\mu_0$ the constant in \ref{item:old hypotheses E}.
Set $\tilde C_1:=C$.
If  $|\bar Y-\hat Y|\geq \tilde C_1\,|\bar x-\hat x|$, then 
\[
t_0\leq\sigma\leq\dfrac{\delta}{2}
\]
and
\[
\mu\leq \dfrac{|\bar Y-\hat Y|^{2}}{\sqrt{\tilde C_1}}\leq \dfrac{(\diam(\Sigma))^{2}}{\sqrt {\tilde C_1}}\leq \mu_0.
\]

Let $Y\in\Sigma$ and $|Y-Y(\lambda)|\leq\mu$ for some $\frac{1}{4}\leq\lambda\leq\frac{3}{4}$.
We will show that 
\begin{equation}\label{eq:Y in subdifferential of u small ball}
Y\in\partial u\(B(x_0,\sigma)\cap S^{n-1}\).
\end{equation}
Notice that $B(x_0,\sigma)\cap S^{n-1}\subseteq B_{2\delta}\cap S^{n-1}\subseteq\Omega$, and if $|x-x_0|\geq\sigma$ and $x\in\Omega$, then $u(x)\geq h(x,Y,X_0)$. 
If  $X=u(x)x$, then this implies that $X$ is outside the region enclosed by the oval $\mathcal O\(Y,|X_0|+\kappa|X_0-Y|\)$ thorough $X_0$ and focus $Y$  which implies that the oval through $X$ with focus $Y$ encloses 
$\mathcal O\(Y,|X_0|+\kappa|X_0-Y|\)$.
Therefore $|X|+\kappa|X-Y|\geq |X_0|+\kappa|X_0-Y|$ for $|x-x_0|\geq\sigma$ and $x\in\Omega$, and
by continuity 
\[
\inf\{|X|+\kappa|X-Y|:X=u(x)x, x\in\Omega\}=|\tilde X|+\kappa|\tilde X-Y|
\]
for some $\tilde X=u(\tilde x)\tilde x$ with $\tilde x\in \bar B(x_0,\sigma)\cap S^{n-1}$.
So each $X=u(x)x$, with $x\in \Omega$, is outside the interior of the region enclosed by oval $\mathcal O\(Y,|\tilde X|+\kappa|\tilde X-Y|\)$
which implies that $u(x)\geq h(x,Y,\tilde X)$, for all $x\in\Omega$. Since $u(\tilde x)=|\tilde X|$ we obtain that
$Y\in\partial u(\tilde x)$ and \eqref{eq:Y in subdifferential of u small ball} is proved.

Therefore
\[
N_{\mu}\(\left\{[\bar Y,\hat Y]_{X_0}:\frac{1}{4}\leq\lambda\leq\frac{3}{4}\right\}\)\cap\Sigma\subset\partial u(B(x_0,\sigma)\cap S^{n-1}).
\]
Taking $H^{n-1}$-measures on both sides, using \ref{item:old hypotheses E} on the left hand side and \eqref{measure} on the right hand side yields
\[
C_{\star}\mu^{n-2}|\bar Y-\hat Y|\leq C^{\star}\sigma^{n-1}
\]
which from the definitions of $\mu$ and $\sigma$ implies
$
|\bar Y-\hat Y|\leq \tilde C_2\, |\bar x-\hat x|^{\alpha},
$
with $\tilde C_2$ an structural constant.
\end{proof}

We can now deduce H\"older estimates for the gradients of refractors.

\begin{theorem}\label{thm:c1alpha estimates for general refractors}
If \ref{item:hypotheses A}, \ref{item:old hypotheses C}, \ref{item:old hypotheses D}, and \ref{item:old hypotheses E} hold, and $u$ is a refractor from $\Omega$ to $\Sigma$ in the sense of Definition \ref{def:definition of refractor} satisfying \eqref{measure}, then $u\in C_{\text{\rm loc}}^{1,\alpha}(\Omega)$.
\end{theorem}
\begin{proof}
Let $x_0\in\Omega$. 
We first show that $\partial u(x_0)$ is singleton. Fix $\delta>0$ such that $B(x_0,2\delta)\cap S^{n-1}\subseteq\Omega$ and suppose $Y_0,Y_1\in \partial u(x_0)$, with $Y_1\neq Y_0$. Let $\bar x\in B(x_0,\delta)\cap S^{n-1}$ and $\bar Y\in\partial u(\bar x)$.
By Theorem \ref{thm:main}, $|\bar Y-Y_0|\leq C|\bar x-x_0|^{\alpha}$ and $|\bar Y-Y_1|\leq C|\bar x-x_0|^{\alpha}$ where the constant $C$ depends on $\delta$. Hence, $|Y_1-Y_0|\leq 2C|\bar x-x_0|^{\alpha}$, so if we take $\bar x$ close enough to $x_0$ we get a contradiction.

Let $Y\in\partial u(x_0)$. We first claim that for any $\eta\perp x_0$, $|\eta|=1$, we have $D_{\eta}u(x_0)=\langle \nabla h(x_0,Y,X_0),\eta\rangle$, where $X_0=u(x_0)x_0$.
To see this, let $c$ be any curve  such that $c(0)=x_0$ and $c^{\prime}(0)=\eta$ and $c(t)\in B(x_0,\delta)\cap S^{n-1}$ for all $t$ near $0$.
Since $u$ is a refractor
\[
u(c(t))-u(x_0)\geq h(c(t),Y,X_0)-h(x_0,Y,X_0)
\]
for all $t$ near $0$.
Let $Y(t)\in \partial u(c(t))$ and $X(t)=u(c(t))c(t)$. Since $u(x)\geq h(x,Y(t),X(t))$ for all $x\in\Omega$, we get
\[
u(x_0)-u(c(t))\geq h(x_0,Y(t),X(t))-h(c(t),Y(t),X(t))
\]
for all $t$ near zero.
Therefore, we have for all $t>0$ small
\[
\dfrac{h(c(t),Y,X_0)-h(x_0,Y,X_0)}{t}\leq \dfrac{u(c(t))-u(x_0)}{t}\leq \dfrac{h(c(t),Y(t),X(t))-h(x_0,Y(t),X(t))}{t}.
\]
Note that for each $t$
\[
\dfrac{h(c(t),Y(t),X(t))-h(x_0,Y(t),X(t))}{t}=\langle\nabla h(\tilde x,Y(t),X(t)),\frac{c(t)-c(0)}{t}\rangle
\]
for some $\tilde x\in[x_0,c(t)]$.
From Theorem \ref{thm:main}, $Y(t)\rightarrow Y$ as $t\rightarrow 0$,  and
$X(t)\rightarrow X_0$ by continuity of $u$.
Letting $t\rightarrow 0$ the claim follows. 

Define $\tilde u(X)=u\(X/|X|\)$ for $X$ with $X/|X|\in \Omega$. We will show that for each $x_0\in \Omega$
\begin{equation}\label{eq:tangential gradient}
\nabla\tilde u(x_0)=\nabla^{T}h(x_0,Y,X_0):=\nabla h(x_0,Y,X_0)-\langle\nabla h(x_0,Y,X_0),x_0\rangle x_0.
\end{equation}
Indeed, let $c(t)=\dfrac{x_0+te_i}{|x_0 +te_i|}$ and notice that $c(0)=x_0$ and $c^{\prime}(0)=e_i-\langle x_0,e_i\rangle x_0$.
Since $\dfrac{\tilde u(x_0+te_i)-\tilde u(x_0)}{t}=\dfrac{u(c(t))-u(x_0)}{t}$, letting $t\rightarrow 0$ and using the first part we get
\[
\dfrac{\partial\tilde u}{\partial x_i}(x_0)=\langle\nabla h(x_0,Y,X_0),e_i-\langle x_0,e_i\rangle x_0\rangle
\]
and the desired formula follows.

Next,  let $\bar x,\hat x\in B(x_0,\delta)\cap S^{n-1}\subset \Omega$, and let $\bar Y\in\partial u(\bar x)$ and $\hat Y\in\partial u(\hat x)$. We shall prove that
\begin{equation}\label{eq:holder estimate of gradient tilde u}
|\nabla\tilde u(\bar x)-\nabla\tilde u(\hat x)|\leq C\,|\bar x-\hat x|^{\alpha}.
\end{equation}
First notice that 
\[
|\nabla ^{T}h(\bar x,\bar Y,\bar X)-\nabla ^{T}h(\hat x,\hat Y,\hat X)|\leq 2\,|\nabla h(\bar x,\bar Y,\bar X)-\nabla h(\hat x,\hat Y,\hat X)| +C\,|\bar x-\hat x|,
\]
since $|\nabla h(\bar x,\bar Y,\bar X)|$ is bounded.
Next write
\begin{align*}
|\nabla h(\bar x,\bar Y,\bar X)-\nabla h(\hat x,\hat Y,\hat X)|
&\leq |\nabla h(\bar x,\bar Y,\bar X)-\nabla h(\bar x,\hat Y,\bar X)|\\
&\qquad +
|\nabla h(\bar x,\hat Y,\bar X)-\nabla h(\hat x,\hat Y,\bar X)|\\
&\qquad \qquad +
|\nabla h(\hat x,\hat Y,\bar X)-\nabla h(\hat x,\hat Y,\hat X)|.
\end{align*}
First, by Lemma \ref{aux6} $|\nabla h(\bar x,\bar Y,\bar X)-\nabla h(\bar x,\hat Y,\bar X)|\leq C\,|\bar Y-\hat Y|$.
Second, that $|\nabla h(\bar x,\hat Y,\bar X)-\nabla h(\hat x,\hat Y,\bar X)|\leq C|\bar x-\hat x|$ follows using the mean value theorem in $x$ from the estimates in the proof of Lemma \ref{aux7}, i.e., from \eqref{eq:lower estimated of Delta in square root}, \eqref{eq:definition of Deltat} and \eqref{eq:partial h with respect to xi}.
For the third term,
from \eqref{eq:partial h with respect to xi} we can write
\begin{align*}
&\nabla h(\hat x,\hat Y,\bar X)-\nabla h(\hat x,\hat Y,\hat X)\\
&=
\dfrac{\kappa^{2}\,h(\hat x,\hat Y,\bar X)\hat Y}{\sqrt{\(\bar b-\kappa^{2}\, x\cdot \hat Y\)^{2}-(1-\kappa^{2})\(\bar b^{2}-\kappa^{2}|\hat Y|^{2}\)}}
-\dfrac{\kappa^{2}\,h(\hat x,\hat Y,\hat X)\hat Y}{\sqrt{\(\hat b-\kappa^{2}\, x\cdot \hat Y\)^{2}-(1-\kappa^{2})\(\hat b^{2}-\kappa^{2}|\hat Y|^{2}\)}},
\end{align*}
where $\bar b=|\bar X|+\kappa|\hat Y-\bar X|$ and $\hat b=|\hat X|+\kappa|\hat Y-\hat X|$. 
Since $\hat Y\in\Sigma$ and $\bar X,\hat X\in \Gamma_{C_1C_2}$, and noticing that $|\bar b-\hat b|\leq C_\kappa\,|\bar X-\hat X|$, it follows from Definitions \eqref{eq:polar radius of the oval}, \eqref{eq:definition of function h}, and Lemma \ref{aux4} that 
\[
|\nabla h(\hat x,\hat Y,\bar X)-\nabla h(\hat x,\hat Y,\hat X)|\leq C\,|\bar X-\hat X|.
\] 
Therefore
\[
|\nabla\tilde u(\bar x)-\nabla\tilde u(\hat x)|
\leq 
C\,\(|\bar x-\hat x|+|\bar Y-\hat Y|+ |\bar X-\hat X|\).
\]
We also have 
$|\bar X-\hat X|=|u(\bar x)\bar x-u(\hat x)\hat x|\leq C_1\,|\bar x-\hat x|+|u(\bar x)-u(\hat x)|\leq C\,|\bar x-\hat x|$, since $u$ is Lipschitz.
From Theorem \ref{thm:main} we then obtain \eqref{eq:holder estimate of gradient tilde u} and the proof is complete.

\end{proof}

\subsection{Regularity of weak solutions}

We now apply Theorem \ref{thm:c1alpha estimates for general refractors} to show that weak solutions to the near field refractor problem defined with the tracing map are $C_{loc}^{1,\alpha}$.
Existence of weak solutions is proved in \cite{gutierrez-huang:nearfieldrefractor}.

Recall that the tracing mapping $\mathcal T_u$ is defined as follows:
given $Y\in\Sigma$, $\mathcal T_u(Y)=\{x\in \Omega:Y\in\partial u(x)\}$.
A weak solution $u$ to the refractor problem from $\Omega$ to $\Sigma$ satisfies 
\begin{equation}\label{eq:weak refractor solution}
\mu(\mathcal T_u(B))=\nu(B),\qquad \text{ for all Borel $B\subset\Sigma$.}
\end{equation}
Here $\mu=f(x)\,dx$ with $f\in L^1(\Omega)$, $f>0$ a.e., and $\nu$ is a measure on the target $\Sigma$ so that  the energy conservation condition $\int_\Omega f(x)\,dx=\nu(\Sigma)$ holds.

\begin{theorem}\label{thm:regularity of Brenier solutions}
Assume that \ref{item:hypotheses A}, \ref{item:old hypotheses C}, \ref{item:old hypotheses D}, and \ref{item:old hypotheses E} hold and the target $\Sigma$ is differentiable.
If $f\in L^\infty(\Omega)$, $\nu\ll H^{n-1}$, and $H^{n-1}=g\,d\nu$ with $0\leq g(x)\leq \alpha$ for a.e. $x\in\Sigma$, then each weak solution $u$ to \eqref{eq:weak refractor solution} satisfies \eqref{measure}, and therefore from Theorem \ref{thm:c1alpha estimates for general refractors} $u\in C^{1,\alpha}_{loc}$.
\end{theorem}

\begin{proof}
Since $\Sigma$ is differentiable, then the visibility condition implies that the tangent plane to $\Sigma$ at each point cannot intersect the interior of $\Gamma_{c_1,c_2}$.
Indeed, suppose the tangent plane $T_{Y}$ to $\Sigma$ at $Y$ intersects $\Gamma_{c_1,c_2}$ at $X_0$ and with a ball $B(X_0,\epsilon)\subset \Gamma_{c_1,c_2}$.
The segment from $X_0$ to $Y$ is on $T_{Y}$ and by visibility for each $X\in B(X_0,\epsilon)$, the segment from $X$ to $Y$ intersects $\Sigma$ only at $Y$. 
This implies that $\Sigma$ cannot be differentiable at $Y$, because if $\Sigma$ were differentiable at $Y$, then $T_Y\cap C=\{Y\}$ with $C$ the cone with vertex $Y$ and base $B(X_0,\epsilon)$, but $\overline{X_0Y}\subset 
T_Y\cap C$.

Now let 
\[
S^{\star}=\{Y\in\Sigma:Y\in\partial u(\bar x)\cap\partial u(\hat x),\:\bar x\neq\hat x\in\Omega\}.
\]
We shall prove that $H^{n-1}(S^{\star})=0$.

Define $u^{\star}:\R^{n}\rightarrow \R$ by 
\[
u^{\star}(Y)=\min\left\{|X|+\kappa|X-Y|:X=u(x)x,\;x\in\Omega\right\}.
\]
It is easy to see that $u^{\star}$ is Lipschitz in $\R^{n}$.

If $\bar Y\in\partial u(\bar x)$, $X=u(x)x$, $\bar X=u(\bar x)\bar x$, with $u$ a refractor, then
$X$ is outside the interior of the region enclosed by the oval $\mathcal O(\bar Y,b)$ with $b=|\bar X|+\kappa\,|\bar X-\bar Y|$. This means that the region enclosed by an oval passing through $X$ with focus $\bar Y$ contains
$\mathcal O(\bar Y,b)$, that is,  
\[
|X|+\kappa|X-\bar Y|\geq |\bar X|+\kappa|\bar X-\bar Y|.
\]
Hence
\[
u^{\star}(\bar Y)=|\bar X|+\kappa|\bar X-\bar Y|
\] 
and so
\[
u^{\star}(Y)\leq u^{\star}(\bar Y)+\kappa|\bar X-Y|-\kappa|\bar X-\bar Y|,
\]
for all $Y\in \R^{n}$.
In particular, if $Y_0\in S^{\star}$ and say $Y_0\in\partial u(\bar x)\cap\partial u(\hat x)$, we then have
\[
u^{\star}(Y)\leq u^{\star}(Y_0)+\kappa|\bar X-Y|-\kappa|\bar X-Y_0|,
\]
and
\[
u^{\star}(Y)\leq u^{\star}(Y_0)+\kappa|\hat X-Y|-\kappa|\hat X-Y_0|,
\]
$\hat X=u(\hat x)\hat x$, for all $Y\in \R^{n}$.

Let $O\subseteq \R^{n-1}$ be open and let $\psi:\R^{n-1}\rightarrow \R^{n}$ be Lipschitz such that $\Sigma=\psi(\bar O)$ and $\psi$ is one to one in $\bar O$.
Set $\tilde{S}=\psi^{-1}\(S^{\star}\)$. We show that $H^{n-1}(\tilde{S})=0$.

Define $h(Y^{\prime})=u^{\star}(\psi(Y^{\prime}))$.
Since $u^\star$ is Lipschitz, $h$ is Lipschitz in $\R^{n-1}$. We claim that $h$ is not differentiable in $\tilde S$.
Let $Y_0^{\prime}\in \tilde{S}$, so $Y_0=\psi(Y_0^{\prime})\in S^\star$, that is, there are $\bar x\neq \hat x$ in $\Omega$ with $Y_0\in\partial u(\bar x)\cap\partial u(\hat x)$.
Then
$
h(Y^{\prime})\leq h(Y_0^{\prime})+|\bar X-\psi(Y^{\prime})|-|\bar X-\psi(Y_0^{\prime})|
$ and 
$
h(Y^{\prime})\leq h(Y_0^{\prime})+|\hat X-\psi(Y^{\prime})|-|\hat X-\psi(Y_0^{\prime})|
$ 
for all $Y'\in \R^n$ with $\bar X=u(\bar x)\bar x$ and $\hat X=u(\hat x)\hat x$. 
If $h$ were differentiable at $Y_0^{\prime}$, then we would have 
\[
\nabla_{Y^{\prime}}(|\bar X-\psi(Y^{\prime})|)=\nabla_{Y^{\prime}}(|\hat X-\psi(Y^{\prime})|)
\]
at $Y^{\prime}=Y_0^{\prime}$.
Thus 
$
D\psi(Y_0^{\prime})^{T}\dfrac{Y_0-\bar X}{|Y_0-\bar X|}=D\psi(Y_0^{\prime})^{T}\dfrac{Y_0-\hat X}{|Y_0-\hat X|}.
$
Letting $w=\dfrac{Y_0-\bar X}{|Y_0-\bar X|}-\dfrac{Y_0-\hat X}{|Y_0-\hat X|}$, yields
$D\psi(Y_0^{\prime})^{T}w=0$. 
If $v_{k}$ denote the columns of $D\psi(Y_0^{\prime})$, this means that $\langle v_{k},w\rangle=0$, for $1\leq k\leq n-1$. 
Since the $v_{k}$'s span the tangent plane to $\Sigma$ at $Y_0$, we get that
$w$ is normal to the tangent plane to $\Sigma$ at $Y_0$.
In particular, the line $Y_0+t\,\(\dfrac{Y_0-\bar X}{|Y_0-\bar X|}+\dfrac{Y_0-\hat X}{|Y_0-\hat X|}\)$ is contained in the tangent plane to $\Sigma$ at $Y_0$. 
But it is easy to see that this line intersects the straight segment $[\bar X,\hat X]$, which implies that either both $\bar X$ and $\hat X$ are on the tangent plane or they are on opposite sides of the tangent plane. In either case, since $\bar X$ and $\hat X$ are on the graph of $u$, the tangent plane intersects the graph of $u$, which contradicts our initial assumption.

Since $h$ is Lipschitz we obtain that $H^{n-1}(\tilde{S})=0$.
This implies, since $\psi$ is Lipschitz, that $H^{n-1}(S^{\star})=0$ as we wanted to show.

From the assumption, $\nu \ll H^{n-1}$, we will show first that $\nu(\partial u(B))\leq \mu(B)$ for each $B\subset\Omega$ Borel set.
If
\[
S=\{x\in\Omega:\text{there exists $\bar x\neq x$, $\bar x\in \Omega$ such that $\partial u(x)\cap\partial u(\bar x)\neq\emptyset$}\},
\]
let us see that $\mu(S)=0$.
Indeed, since $\mathcal T_u(S^{\star})=S$, from the definition of weak solution $\nu(S^{\star})=\mu(\mathcal T_u(S^{\star}))=\mu(S)$. Since 
$H^{n-1}(S^{\star})=0$, we then get $\mu(S)=0$.
On the other hand, $\mathcal T_u(\partial u(B))\subset B\cup S$ so
$\nu(\partial u(B))=\mu(\mathcal T_u(\partial u(B)))\leq\mu(B\cup S)\leq \mu(B)$ and we are done.

Therefore, to conclude the proof of the theorem, we prove that $u$ verifies \eqref{measure}.
Indeed, for each ball $B_\sigma$ with $B_{\sigma}\cap S^{n-1}\subset \Omega$ we have
\begin{align*}
H^{n-1}\(\partial u\(B_{\sigma}\cap S^{n-1}\)\)
&\leq
\alpha\,\nu\(\partial u\(B_{\sigma}\cap S^{n-1}\)\)
\leq
\alpha\,\mu\(B_{\sigma}\cap S^{n-1}\)\\
&\leq 
\alpha \,\|f\|_\infty \text{surface area}\(B_{\sigma}\cap S^{n-1}\)
\leq 
C\,\sigma^{n-1}.
\end{align*}

\end{proof}

\providecommand{\bysame}{\leavevmode\hbox to3em{\hrulefill}\thinspace}
\providecommand{\MR}{\relax\ifhmode\unskip\space\fi MR }
\providecommand{\MRhref}[2]{%
  \href{http://www.ams.org/mathscinet-getitem?mr=#1}{#2}
}
\providecommand{\href}[2]{#2}


\begin{thebibliography}{MTW05}

\bibitem[AG17]{abedin-gutierrez:numericalgeneratedjacobians}
F.~Abedin and C.~E. Guti\'errez, \emph{An iterative method for generated
  {J}acobian equations}, Calc. Var. PDEs \textbf{56} (2017), no.~101, 1--14.

\bibitem[AGT16]{abedin-gutierrez-tralli:parallelrefractor}
F.~Abedin, C.~E. Guti\'errez, and G.~Tralli, \emph{${C}^{1,\alpha}$-estimates
  for the parallel refractor}, Nonlinear Analysis \textbf{142} (2016), 1--25.

\bibitem[CGH08]{caffarelli-gutierrez-huang:antennaannals}
L.~A. Caffarelli, C.~E. Guti\'errez, and Qingbo Huang, \emph{On the regularity
  of reflector antennas}, Ann. of Math. \textbf{167} (2008), 299--323.

\bibitem[GH14]{gutierrez-huang:nearfieldrefractor}
C.~E. Guti\'errez and Qingbo Huang, \emph{The near field refractor}, Annales de
  l'Institut Henri Poincar{\'e} (C) Analyse Non Lin{\'e}aire \textbf{31}
  (2014), no.~4, 655--684.

\bibitem[GK17]{2017-guillen-kitagawa:generatedjacobianeqs}
N\'estor Guill\'en and Jun Kitagawa, \emph{Pointwise estimates and regularity
  in geometric optics and other generated jacobian equations}, Comm. Pure App.
  Math. \textbf{70} (2017), no.~6, 1146--1220.

\bibitem[GM19]{2019-gutierrez-mawi:numericalsolutionnearfield}
C.~E. Guti\'errez and H.~Mawi, \emph{On the numerical solution of the near
  field refractor problem}, Preprint, 2019.

\bibitem[GT15]{gutierrez-tournier:REGULARITYFORTHENEARFIELDPARALLELREFRACTORANDREFLECTORPROBLEMS}
C.~E. Guti\'errez and F.~Tournier, \emph{Regularity for the near field parallel
  refractor and reflector problems}, Calc. Var. PDEs \textbf{54} (2015), no.~1,
  917--949.

\bibitem[KM10]{2010-kim-mccann:DMASM}
Young-Heon Kim and Robert~J. McCann, \emph{Continuity, curvature, and the
  general covariance of optimal transportation}, J. Eur. Math. Soc. (JEMS)
  \textbf{12} (2010), no.~4, 1009--1040.

\bibitem[LGM17]{deleo-gutierrez-mawi:numericalrefractor}
R.~De Leo, C.~E. Guti\'errez, and H.~Mawi, \emph{On the numerical solution of
  the far field refractor problem}, Nonlinear Analysis: Theory, Methods \&
  Applications \textbf{157} (2017), 123--145.

\bibitem[Loe09]{loeper:actapaper}
G.~Loeper, \emph{On the regularity of solutions of optimal transportation
  problems}, Acta Math. \textbf{202} (2009), 241--283.

\bibitem[MTW05]{MaTrudingerWang:regularityofpotentials}
Xi-Nan Ma, N.~Trudinger, and Xu-Jia Wang, \emph{Regularity of potential
  functions of the optimal transportation problem}, Arch. Rational Mech. Anal.
  \textbf{177} (2005), no.~2, 151--183.

\end{thebibliography}
\end{document}